\documentclass[final]{siamltex}

\usepackage{epsf, graphicx}
\usepackage{latexsym,amsfonts,amsbsy,amssymb}
\usepackage{amsmath}
\usepackage{float}
\usepackage{color}
\allowdisplaybreaks 

\allowdisplaybreaks 
\newtheorem{remark}{Remark}[section]
\newtheorem{algorithm}{Algorithm}[section]
\title{A Multilevel Correction Scheme for Nonsymmetric Eigenvalue Problems by Finite Element Methods}
\author{Hehu Xie\thanks{LSEC, ICMSEC, NCMIS, Academy of Mathematics and
Systems Science, Chinese Academy of Sciences, Beijing 100190, China
({\tt hhxie@lsec.cc.ac.cn}).}
\and Zhimin Zhang\thanks{Beijing Computational Science Research Center, and
Department of Mathematics, Wayne State University, Detroit, MI 48202, USA
({\tt zzhang@math.wayne.edu}).}}
\begin{document}
\maketitle
\begin{abstract}
A multilevel correction scheme is proposed to solve {\color{black}defective and nodefective of nonsymmetric partial differential
operators} by the finite element method. The method includes multi correction steps in a sequence of finite
element spaces. In each correction step, we only need to solve two source problems on a finer
finite element space and two eigenvalue problems on the coarsest finite element space.
The accuracy of the eigenpair approximation is improved after each correction step.
This correction scheme improves overall efficiency of the finite element method in
solving nonsymmetric eigenvalue problems.
\end{abstract}
\pagestyle{myheadings}
\thispagestyle{plain}
\markboth{Hehu Xie and Zhimin Zhang}{Multilevel correction for nonsymmetric eigenvalue}

\section{Introduction}

As we know, the numerical approximation of eigenvalue problems plays a
central role in the analysis of the stability for nonlinear partial differential
equations. For example in fluid mechanics, the analysis of the
hydrodynamic stability always leads to a nonsymmetric eigenvalue problems (see
\cite{CliffeHallHouston,HeuvelineRannacher_2001,HeuvelineRannacher_2006}).
The stability of the underlying flow depends on the real part of the eigenvalue which has the
smallest real part (see \cite{CliffeHallHouston,HeuvelineRannacher_2006}). For more details, please
refer \cite{CliffeHallHouston,HeuvelineRannacher_2006,SchmidHenningson}.
The aim of understanding the stability of nonlinear partial differential equations naturally leads to
the computation of the eigenvalue problems with some numerical methods. The main content of this paper is
to design an efficient finite element method to compute nonsymmetric eigenvalue problems.

Recently, a multigrid method is designed to solve the self-adjoint eigenvalue problem based
on a type of multilevel correction method \cite{LinLuoXie,LinXie,LinXie_Multigrid,Xie_IMA}.
But as we know, the analysis of the stability
for nonlinear partial differential equations always leads to nonsymmetric eigenvalue problems
\cite{CliffeHallHouston,HeuvelineRannacher_2006}
 and the extensions of the multilevel method for self-adjoint eigenvalue problems to the nonsymmetric
ones is not direct \cite{Kolman,XuZhou,YangFan} and needs more analysis.
So the purpose of this paper is to propose a multilevel
correction scheme to solve nonsymmetric eigenvalue problems based on the finite element method.
In the past, a two-grid finite element method was proposed and analyzed by Xu and Zhou in
\cite{XuZhou} for symmetric eigenvalue problems. Latter, Kolman used
this idea to design a two-level method for nonsymmetric
 eigenvalue problems in \cite{Kolman}. Yang and Fan \cite{YangFan}
 also studied a two-grid method for nonsymmetric eigenvalue problems.
As an alternative approach, in \cite{NagaZhang,NagaZhangZhou,WuZhang}, the authors used a
 recovery technique PPR to improve the convergence rate for both
 symmetric and nonsymmetric eigenvalue problems. {\color{black}All these methods are designed for the
 nonsymmetric eigenvalue problems under the assumption that the ascent of the concerned eigenvalues is only one
 which means the algebraic eigenspace is the same as the geometric eigenspace.}

Along the line of multilevel correction method, here we present a multilevel correction
scheme to solve nonsymmetric eigenvalue problems {\color{black}without the ascent assumption}.
With the proposed method, solving nonsymmetric eigenvalue problems will not be much more expensive
than solving corresponding source problems. The correction method for eigenvalue problems
in this paper is based on a series of finite element spaces with different levels of accuracy
 which are related to the multilevel method (cf.  \cite{Xu}).

The standard Galerkin finite element method for nonsymmetric eigenvalue problems
has been extensively investigated, e.g. Babu\v{s}ka and Osborn
\cite{Babuska2,BabuskaOsborn}, Chatelin \cite{Chatelin} and
references cited therein.  Here we adopt some basic results in these
papers to { carry on error estimates for our multilevel
correction scheme}. It will be shown that the convergence rate of
the eigenpair approximations can be improved after each correction step.

Our multilevel correction procedure can be described as follows: (1)\
solve an eigenvalue problem in the coarsest finite element space;
(2)\ solve a source problem in an augmented space with the associated eigenfunction
 from (1) as the load vector; (3)\ solve the eigenvalue problem
again on a finite element space constructed by enhancing
the coarsest finite element space with the eigenfunction obtained in step (2).
Then go to step (2) for the next loop.

In this method, we replace solving the eigenvalue problem in finer
finite element spaces by solving a series of boundary value problems
in a series of nested finite element spaces (with the finest space as the last one) and a series of
eigenvalue problems in the coarsest finite element space;
and yet, we achieve the same accuracy as solving the eigenvalue problem in the finest space.
It is well known that there exist multigrid methods that solve boundary value problems with the
optimal computational work (cf. \cite{Xu_Two_Grid}).
Therefore, combined with the multigrid method, our correction method improves overall efficiency
in solving nonsymmetric eigenvalue problems (cf. \cite{Xie_IMA,Xie_JCP}).

An outline of the paper goes as follows. In Section 2, we introduce the
finite element method for nonsymmetric eigenvalue problems.
An one level correction scheme is described and analyzed in Section 3.
In Section 4, we propose and analyze a multilevel correction
algorithm to solve nonsymmetric eigenvalue problems by the finite element method.
Some numerical examples are presented in Section 5 to validate our
theoretical analysis and some concluding remarks are given in the last section.

\section{Discretization by finite element method}
In this section, we introduce some notation and error estimates of
the finite element approximation for nonsymmetric eigenvalue problems.
Throughout this paper, the letter $C$ (with or without subscripts) denotes a generic
positive constant which may be different at different occurrences.
For convenience, we use symbols $\lesssim$, $\gtrsim$, and $\approx$,
such that $x_1\lesssim y_1, x_2\gtrsim y_2$,
and $x_3\approx y_3$ have meanings: $x_1\leq C_1y_1$, $x_2 \geq c_2y_2$,
and $c_3x_3\leq y_3\leq C_3x_3$, for some constants $C_1, c_2, c_3$,
and $C_3$ that are independent of mesh sizes (cf. \cite{Xu}).

We consider the following eigenvalue problem:

Find $\lambda\in \mathcal{C}$ and $u$ such that
\begin{equation}\label{Eigenvalue_Problem}
\left\{
\begin{array}{rcl}
-\nabla\cdot(A\nabla u)+\mathbf b\cdot\nabla u +\phi u&=&\lambda\varphi u,\ \ \ {\rm in}\ \Omega,\\
u&=&0,\ \ \ \ \ \ \ {\rm on}\ \partial\Omega,\\
\int_{\Omega}\varphi |u|^2d\Omega&=&1,
\end{array}
\right.
\end{equation}
where $\Omega\subset \mathcal{R}^d$ is a bounded polygonal domain, $A\in\mathcal{C}^{d\times d}$,
$\mathbf b\in \mathcal{C}^d$, $\phi$ is a function defined on $\Omega$ and
$\varphi$ is a real positive function with $\varphi\geq c_0>0$.

We define $V:=H_0^1(\Omega)$ with the usual norm $\|\cdot\|_1$.
The corresponding variational form of (\ref{Eigenvalue_Problem}) can be stated as follows:

Find $(\lambda,u)\in \mathcal{C}\times V$ such that $b(u,u)=1$ and
\begin{eqnarray}\label{Eigenvalue_Problem_Weak}
a(u,v)&=&\lambda b(u,v),\ \ \ \ \ \forall v\in V,
\end{eqnarray}
where
\begin{eqnarray*}
a(u,v)&=&(A\nabla u,\overline{\nabla v})+(\mathbf b\cdot \nabla u, \bar{v})+(\phi u,\bar{v}),\\
b(u,v)&=&(\varphi u,\bar{v})
\end{eqnarray*}
with $(\cdot,\cdot)$ denoting the inner product in the space $L^2(\Omega)$.
The corresponding adjoint eigenvalue problem is:

Find $(\lambda,u^*)\in \mathcal{C}\times V$ such that $b(u^*,u^*)=1$ and
\begin{eqnarray}\label{Eigenvalue_Problem_Weak_Adjoint}
a(v,u^*)&=&\lambda b(v,u^*),\ \ \ \ \ \forall v\in V.
\end{eqnarray}
In the sequel, we also use the norm $\|v\|_b=\sqrt{b(v,v)}$ which is equivalent to the
$L^2(\Omega)$ norm $\|\cdot\|_0$. Here the bilinear form $a(\cdot,\cdot)$ is assumed
to satisfy
\begin{eqnarray}
\|w\|_1 \lesssim \sup_{{ v}\in V}\frac{a(w,v)}{\|v\|_1}\ \
{\rm and}\ \ \|w\|_1\lesssim \sup_{v\in V}\frac{a(v,w)}{\|v\|_1},\ \ \forall w\in V.
\end{eqnarray}
We further assume { that $a(\cdot,\cdot)$ is} $V$-elliptic, i.e.,
\begin{eqnarray}\label{Ellipticity_Of_a}
\|u\|_1^2&\lesssim & {\rm Re}\ a(u,u),\ \ \ \ \forall u\in V.
\end{eqnarray}

\subsection{Operator reformulation}
We introduce the operators $T,\ T_*\in \mathcal{L}(V)$ defined by the equation
\begin{eqnarray}
a(Tu,v)=b(u,v)={a}(u,T_*v),\ \ \ \ \ \forall u, v\in V.
\end{eqnarray}
The eigenvalue problem (\ref{Eigenvalue_Problem_Weak}) can be written as an operator form
 for $\lambda\neq 0$ (denoting $\mu:=\lambda^{-1}$):
\begin{eqnarray}
Tu=\mu u,
\end{eqnarray}
{ with}
\begin{eqnarray}
T_*u^*=\bar{\mu}u^*
\end{eqnarray}
for the adjoint eigenvalue problem (\ref{Eigenvalue_Problem_Weak_Adjoint}).
 Note that ellipticity condition (\ref{Ellipticity_Of_a})
guarantees that every eigenvalue $\lambda$ is nonzero.
It is well known that the operators $T$ and $T_*$ are compact. Thus the spectral theory for compact
operators  gives us a complete characterization of the eigenvalue problem (\ref{Eigenvalue_Problem_Weak}).

There is a countable set of eigenvalues of (\ref{Eigenvalue_Problem_Weak}). Let $\lambda$ be { an
eigenvalue} of problem (\ref{Eigenvalue_Problem_Weak}). There exists a smallest integer $\alpha$
{\color{black}which are called the ascent} such that
\begin{eqnarray}
N((T-\mu)^{\alpha})=N((T-\mu)^{\alpha+1}),
\end{eqnarray}
where $N$ denotes the null space and we use the notation $\mu=\lambda^{-1}$.
Let $M(\lambda)=M_{\lambda,\mu}=N((T-\mu)^{\alpha})$
and $Q(\lambda)=Q_{\lambda,\mu}=N(T-\mu)$ denote the algebraic and geometric eigenspaces, respectively.
The subspaces $Q(\lambda)\subset M(\lambda)$ are finite dimensional. The numbers
$m={\rm dim}M(\lambda)$ and $q={\rm dim}Q(\lambda)$ are called the algebraic and the geometric multiplicities
of $\mu$ (and $\lambda$). The vectors in $M(\lambda)$ are generalized eigenvectors. The order of a generalized
eigenvector is the smallest integer $j$ such that $(T-\mu)^ju=0$ (vectors in $Q(\lambda)$ being
generalized eigenvectors of order $1$). Let us point out that a generalized eigenvector $u^j$ of order $j$ satisfies
\begin{eqnarray}\label{Iteration_Scheme}
a(u^j,v)&=&\lambda b(u^j,v)+\lambda a(u^{j-1},v),\ \ \ \ \forall v\in V,
\end{eqnarray}
where $u^{j-1}$ is a generalized eigenvector of order $j-1$.

Similarly we define the spaces of (generalized) eigenvectors for the adjoint problem
\begin{eqnarray*}
M^*(\lambda)=M_{\lambda,\mu}^*=N((T_*-\bar{\mu})^{\alpha})\ \ \ {\rm and}\ \ \
Q^*(\lambda)=Q^*_{\lambda,\mu}=N(T_*-\bar{\mu}).
\end{eqnarray*}

Note that $\mu$ is an eigenvalue of $T$ ($\lambda$ is an eigenvalue of problem
(\ref{Eigenvalue_Problem_Weak})) if and only if
$\bar{\mu}$ is an eigenvalue of $T_*$ ($\lambda$ is an eigenvalue of adjoint problem
(\ref{Eigenvalue_Problem_Weak_Adjoint})) with the ascent $\alpha$
and the algebraic multiplicity $m$ for both eigenvalues being the same.

\subsection{Galerkin discretization}
Now, let us define the finite element approximations for the problem
(\ref{Eigenvalue_Problem_Weak}). First we generate a shape-regular
decomposition of the computing domain $\Omega\subset \mathcal{R}^d\
(d=2,3)$ into triangles or rectangles for $d=2$ (tetrahedrons or
hexahedrons for $d=3$). The diameter of a cell $K\in\mathcal{T}_h$
is denoted by $h_K$. The mesh diameter $h$ describes the maximum
diameter of all cells $K\in\mathcal{T}_h$. Based on the mesh
$\mathcal{T}_h$, { we construct} a finite element space denoted by
$V_h\subset V$. In order to { define our} multilevel correction method, we
start the process on { an initial mesh $\mathcal{T}_H$ with mesh
size $H$ and the initial finite element space $V_H$
defined on} $\mathcal{T}_H$.
In this paper, the finite element space $V_h$ is assumed to satisfy
\begin{eqnarray}\label{Inf_Sup_Discrete}
\|w_h\|_1 \lesssim \sup_{v_h\in V_h}\frac{a(w_h,v_h)}{\|v_h\|_1}\ \ {\rm and}
\ \ \|w_h\|_1\lesssim \sup_{v_h\in V_h}\frac{a(v_h,w_h)}{\|v_h\|_1},\ \ \forall w_h\in V_h.
\end{eqnarray}

The standard Galerkin discretization of the problem (\ref{Eigenvalue_Problem_Weak}) is the following:

Find $(\lambda_h,u_h)\in \mathcal{C}\times V_h$ such that $b(u_h,u_h)=1$ and
\begin{eqnarray}\label{Eigenvalue_Problem_Weak_Discrete}
a(u_h,v_h)&=&\lambda_h b(u_h,v_h),\ \ \ \ \ \forall v_h\in V_h.
\end{eqnarray}
Similarly, the discretization of the adjoint problem (\ref{Eigenvalue_Problem_Weak_Adjoint})
can be defined as:

Find $(\lambda_h,u_h^*)\in \mathcal{C}\times V_h$ such that $b(u_h^*,u_h^*)=1$ and
\begin{eqnarray}\label{Eigenvalue_Problem_Weak_Discrete_Adjoint}
a(v_h,{u_h^*})&=&\lambda_h b(v_h,{u_h^*}),\ \ \ \ \ \forall v_h\in V_h.
\end{eqnarray}

By introducing Galerkin projections $P_h,\ P_h^*\in \mathcal{L}(V,V_h)$
with the following equations
\begin{eqnarray*}
a(P_hu,v_h)&=&a(u,v_h),\ \ \ \ \  \quad \forall u\in V,\ \forall v_h\in V_h,\\
a(v_h,u)&=&a(v_h,P_h^*u),\ \ \ \ \forall u\in V,\ \forall v_h\in V_h,
\end{eqnarray*}
the equation (\ref{Eigenvalue_Problem_Weak_Discrete}) can be rewritten as an operator
 form with $\mu_h:=\lambda_h^{-1}$
(Note that $P_h$ is a bounded operator),
\begin{eqnarray}
P_hTu_h=\mu_hu_h.
\end{eqnarray}
Similarly for the adjoint problem (\ref{Eigenvalue_Problem_Weak_Discrete_Adjoint}), we have
\begin{eqnarray}
P_h^*T_*u_h^*=\bar{\mu}_hu_h^*.
\end{eqnarray}

\subsection{Spectral approximation of compact operators}
Let $\mu$ be an eigenvalue (with algebraic multiplicity $m$) of the compact operator $T$.
If $T$ is approximated by a sequence of compact operators $T_h$ converging to $T$ in norm, i.e.,
$\lim\limits_{h\rightarrow 0+}{ \|T-T_h\|_1}=0$, then for $h$ sufficiently small $\mu$ is approximated by
 $m$ numerical eigenvalues $\{\mu_{j,h}\}_{j=1,\cdots,m}$ (counted according to their algebraic multiplicities)
of $T_h$, i.e.,
\begin{eqnarray*}
\lim_{h\rightarrow 0+}\mu_{j,h}=\mu\ \ \ \ \ {\rm for}\ j=1,\cdots,m.
\end{eqnarray*}
The space of generalized eigenvectors of $T$ is approximated by the subspace
\begin{eqnarray}
M_h(\lambda)=M_h^{\lambda,\mu}=\sum_{j=1}^mN((T_h-\mu_{j,h})^{\alpha_{\mu_{j,h}}}),
\end{eqnarray}
where $\alpha_{\mu_{j,h}}$ is the smallest integer such that $N((T_h-\mu_{j,h})^{\alpha_{\mu_{j,h}}})=N((T_h-\mu_{j,h})^{\alpha_{\mu_{j,h}}+1})$.
We similarly define the space $Q_h(\lambda)=Q_h^{\lambda,\mu}=\sum_{j=1}^mN(T_h-\mu_{j,h})$ and
counterparts $M_h^*(\lambda)$, $Q_h^*(\lambda)$ for
the adjoint problem .

Now, we describe a computational scheme to produce the
algebraic eigenspace { $M_h(\lambda)$} from the geometric eigenspace
$Q_h(\lambda)=\{u_{1,h},\cdots,u_{q,h}\}$
corresponding to eigenvalues $\{\lambda_{1,h},\cdots,\lambda_{q,h}\}$, which
 converge to the same eigenvalue $\lambda$.

Starting from all eigenfunctions in the geometric eigenspace $Q_h(\lambda)$ (of order $1$),
we use the following recursive process to compute
 algebraic eigenspaces (cf. \cite{Shaidurov})
\begin{equation}
\left\{
\begin{array}{rcl}
a(u_{j,h}^{\ell},v_h)-\lambda_{j,h}b(u_{j,h}^{\ell},v_h)&=&\lambda_{j,h}a(u_{j,h}^{\ell-1},v_h),
\ \ \forall v_h\in V_h,\\
b(u_{j,h}^{\ell}, v_h)&=&0,\ \ \ \quad\quad\quad\quad\quad\ \ \forall v_h\in Q_h({\lambda}),
\end{array}
\right.
\end{equation}
where $\ell\geq 2$, $u_{j,h}^{\ell}$ is the general eigenfunction of order
$\ell$ and $u_{j,h}^1=u_{j,h}\in Q_h({\lambda})$ for $j=1,\cdots,q$.

With the above process, we { generate the algebraic eigenspace}
$$M_h(\lambda)=\{u_{1,h},\cdots, u_{q,h},\cdots, u_{m,h}\}$$
corresponding { to} eigenvalues $\{\lambda_{1,h},\cdots,\lambda_{q,h},\cdots, \lambda_{m,h}\}$,
 which converge to the same eigenvalue $\lambda$.  Similarly, we can produce the
adjoint algebraic eigenspace $M_h^*(\lambda)$ { from} the geometric eigenspace $Q_h^*(\lambda)$.

For two linear spaces $A$ and $B$, we denote
\begin{eqnarray*}
\widehat{\Theta}(A,B) = \sup_{w\in A,\|w\|_1=1}\inf_{v\in B}\|w-v\|_1,\ \
\widehat{\Phi}(A,B) = \sup_{w\in A,\|w\|_b=1}\inf_{v\in B}\|w-v\|_{b},
\end{eqnarray*}
{ and} define gaps between $A$ and $B$ in $\|\cdot\|_1$ as
\begin{eqnarray}
\Theta(A,B)=\max\big\{\widehat{\Theta}(A,B), \widehat{\Theta}(B,A)\big\},
\end{eqnarray}
and in $\|\cdot\|_b$ as
\begin{eqnarray}
\Phi(A,B)=\max\big\{\widehat{\Phi}(A,B), \widehat{\Phi}(B,A)\big\}.
\end{eqnarray}

Before introducing the convergence results of the finite element approximation for
nonsymmetric eigenvalue problems, we define the following notation
\begin{eqnarray}
&&\delta_h(\lambda)=\sup_{{u}\in M(\lambda),\|u\|_1=1}\inf_{v_h\in V_h}\|u-v_h\|_1,\\
&&\delta_h^*(\lambda)=\sup_{u^*\in M^*(\lambda),\|u^*\|_1=1}\inf_{v_h\in V_h}\|u^*-v_h\|_1,\\
&&\rho_h(\lambda)=\sup_{{u}\in M(\lambda),\|u\|_b=1}\inf_{v_h\in V_h}\|u-v_h\|_b,\\
&&\rho_h^*(\lambda)=\sup_{u^*\in M^*(\lambda),\|u^*\|_b=1}\inf_{v_h\in V_h}\|u^*-v_h\|_b,\\
&&\eta_a(h)=\sup_{f\in V,\|f\|_b=1}\inf_{v\in V_h}\|T f-v\|_1,\\
&&\eta_a^*(h)=\sup_{f\in V,\|f\|_b=1}\inf_{v\in V_h}\|T_* f-v\|_1.
\end{eqnarray}
In order to derive { error bounds for} eigenpair approximations in the
weak norm $\|\cdot\|_b$, we need the following error estimates in the weak norm $\|\cdot\|_b$
of the finite element approximation.
\begin{lemma}\label{Negative_norm_estimate_Lemma}
(\cite[Lemma 3.3 and Lemma 3.4]{BabuskaOsborn})
\begin{eqnarray}
\eta_a(h)=o(1),\ \ \ \eta_a^*(h)=o(1)\ \ \ {\rm as}\ h\rightarrow 0,
\end{eqnarray}
and
\begin{eqnarray}\label{Negative_norm_Error}
\rho_h(\lambda)&\lesssim& \eta_a^*(h)\delta_h(\lambda),\\
\rho_h^*(\lambda)&\lesssim& \eta_a(h)\delta_h^*(\lambda).
\end{eqnarray}
\end{lemma}

The following theorem is a basic tool for our error estimates.
\begin{theorem}\label{Error_Estimate_Theorem}(\cite[Section 8]{BabuskaOsborn})
When the mesh size $h$ is small enough, we have
\begin{eqnarray}
&&\Theta(M(\lambda),M_h(\lambda))\lesssim \delta_h(\lambda),\ \ \
\Theta(M^*(\lambda),M^*_h(\lambda))\lesssim \delta^*_h(\lambda),\\
&&\Phi(M(\lambda),M_h(\lambda))\lesssim \rho_h(\lambda),\ \ \
\Phi(M^*(\lambda),M^*_h(\lambda))\lesssim \rho^*_h(\lambda),\\
&&|\lambda-\widehat{\lambda}_{h}|\lesssim \delta_h(\lambda)\delta_h^*(\lambda),
\end{eqnarray}
where $\widehat{\lambda}_h={ \frac{1}{m}}\sum_{j=1}^m\lambda_{j,h}$ with $\lambda_{1,h},\cdots,\lambda_{m,h}$
converging to $\lambda$.
\end{theorem}

%
%
%
%

\section{One correction step}
In this section, we present an one-step correction procedure to improve the
accuracy of the current eigenvalue and eigenfunction approximations.
This correction method contains solving some auxiliary source problems
in a finer finite element space and two eigenvalue problems on a coarse
finite element space.

Assume that we have obtained the algebraic eigenpair approximations
$(\lambda_{j,h_k},u_{j,h_k})\in\mathcal{R}\times V_{h_k}$ and
the corresponding adjoint ones
$(\lambda_{j,h_k},u^*_{j,h_k})\in\mathcal{R}\times V_{h_k}$ for $j=i,\cdots,i+m-1$, where
eigenvalues $\{\lambda_{j,h_k}\}_{j=i}^{i+m-1}$ converge to the desired eigenvalue
$\lambda_i$ of (\ref{Eigenvalue_Problem_Weak}).
Now we introduce a correction step to improve the accuracy of the
current eigenpair approximations.
Let $V_{h_{k+1}}\subset V$ be the conforming finite element space based on a finer
mesh $\mathcal{T}_{h_{k+1}}$ which is produced by refining $\mathcal{T}_{h_k}$ in the regular way.
We start from a conforming linear finite element space $V_H$ on the
coarsest mesh $\mathcal{T}_H$ to design the following one correction step.

\begin{algorithm}\label{Correction_Step}
One Correction Step

\begin{enumerate}

\item  For $j=i,\cdots,i+m-1$ Do

Solve the following two boundary value problems:

Find $\widetilde{u}_{j,h_{k+1}}\in V_{h_{k+1}}$ such that
\begin{eqnarray}\label{aux_problem}
a(\widetilde{u}_{j,h_{k+1}},v_{h_{k+1}})&=&b(u_{j,h_k},v_{h_{k+1}}),\ \
\ \forall v_{h_{k+1}}\in V_{h_{k+1}}.
\end{eqnarray}
Find  $\widetilde{u}^*_{j,h_{k+1}}\in V_{h_{k+1}}$ such that
\begin{eqnarray}\label{aux_problem_Adjoint}
a(v_{h_{k+1}},\widetilde{u}^*_{j,h_{k+1}})&=&b(v_{h_{k+1}},u^*_{j,h_k}),\ \
\ \forall v_{h_{k+1}}\in V_{h_{k+1}}.
\end{eqnarray}
End Do
\item  Define two new finite element spaces
\begin{eqnarray*}
V_{H,h_{k+1}}=V_H\oplus{\rm span}\{\widetilde{u}_{i,h_{k+1}},\cdots,\widetilde{u}_{i+m-1,h_{k+1}}\}
\end{eqnarray*}
and
\begin{eqnarray*}
V^*_{H,h_{k+1}}=V_H\oplus{\rm span}\{\widetilde{u}^*_{i,h_{k+1}},\cdots,\widetilde{u}^*_{i+m-1,h_{k+1}}\}.
\end{eqnarray*}


Solve the following two eigenvalue problems:

Find $(\lambda_{j,h_{k+1}},u_{j,h_{k+1}})\in\mathcal{R}\times V_{H,h_{k+1}}$ such
that $b(u_{j,h_{k+1}},u_{j,h_{k+1}})=1$ and
\begin{eqnarray}\label{Eigen_Augment_Problem}
\hskip-0.4cm a(u_{j,h_{k+1}},v_{H,h_{k+1}})=\lambda_{j,h_{k+1}} b(u_{j,h_{k+1}},v_{H,h_{k+1}}),\
\forall v_{H,h_{k+1}}\in V^*_{H,h_{k+1}}.
\end{eqnarray}
Find $(\lambda_{j,h_{k+1}},u^*_{j,h_{k+1}})\in\mathcal{R}\times V^*_{H,h_{k+1}}$ such
that $b(u^*_{j,h_{k+1}},u^*_{j,h_{k+1}})=1$ and
\begin{eqnarray}\label{Eigen_Augment_Problem_Adjoint}
\hskip -0.4cm a(v_{H,h_{k+1}},u^*_{j,h_{k+1}})={\lambda}_{j,h_{k+1}} b(v_{H,h_{k+1}},u^*_{j,h_{k+1}}),\
\forall v_{H,h_{k+1}}\in V_{H,h_{k+1}}.
\end{eqnarray}

{ \item Choose $2q$ eigenpairs $\{\lambda_{j,h_{k+1}}, u_{j,h_{k+1}}\}_{j=i}^{i+q-1}$ and
$\{\lambda_{j,h_{k+1}}, u^*_{j,h_{k+1}}\}_{j=i}^{i+q-1}$ to define  
 two new geometric eigenspaces}
\begin{eqnarray*}
Q_{h_{k+1}}(\lambda_i)={\rm span}\big\{u_{i,h_{k+1}},\cdots, u_{i+q-1,h_{k+1}}\big\}
\end{eqnarray*}
and
\begin{eqnarray*}
Q^*_{h_{k+1}}(\lambda_i)={\rm span}\big\{u^*_{i,h_{k+1}},\cdots, u^*_{i+q-1,h_{k+1}}\big\}.
\end{eqnarray*}
Based on these two geometric eigencpases, compute the corresponding algebraic eigenspaces
\begin{eqnarray}
M_{h_{k+1}}(\lambda_i)={\rm span}\big\{u_{i,h_{k+1}},\cdots, u_{i+m-1,h_{k+1}}\big\}
\end{eqnarray}
and
\begin{eqnarray}
M^*_{h_{k+1}}(\lambda_i)={\rm span}\big\{u_{i,h_{k+1}},\cdots, u_{i+m-1,h_{k+1}}\big\}.
\end{eqnarray}
\end{enumerate}
{ The final output is:}
\begin{eqnarray*}
&&\big(\{\lambda_{j,h_{k+1}}\}_{j=i}^{i+m-1},M_{h_{k+1}}(\lambda_i),M^*_{h_{k+1}}(\lambda_i)\big)=\nonumber\\
&&\ \ \ \ \quad\quad {\it
Correction}\big(V_H,\{\lambda_{j,h_k}\}_{j=i}^{i+m-1}{,}M_{h_k}(\lambda_i),M^*_{h_k}(\lambda_i),V_{h_{k+1}}\big).
\end{eqnarray*}
\end{algorithm}
\begin{remark}
Since in Step 1 of Algorithm \ref{Correction_Step}, the solving process for the boundary value problems
is independent of each other for different $j$, we can implement them in parallel. {\color{black}Furthermore,
the designing for this algorithm does not need the ascent assumption as in \cite{Kolman,YangFan}}.
\end{remark}

\begin{theorem}\label{Error_Estimate_One_Correction_Theorem}
Assume there exist real numbers $\varepsilon_{h_k}(\lambda_i)$ and $\varepsilon_{h_k}^*(\lambda_i)$ such that
the given eigenpairs $\big(\{\lambda_{j,h_k}\}_{j=i}^{i+m-1}, M_{h_k}(\lambda_i),M^*_{h_k}(\lambda_i)\big)$
in {\it One Correction Step \ref{Correction_Step}} have following error estimates
\begin{eqnarray}
\Theta(M(\lambda_i),M_{h_k}(\lambda_i)) &\lesssim& \varepsilon_{h_k}(\lambda_i),\\
\Theta(M^*(\lambda_i),M^*_{h_k}(\lambda_i)) &\lesssim& \varepsilon^*_{h_k}(\lambda_i),\label{Error_u_h_1}\\
\Phi(M(\lambda_i),M_{h_k}(\lambda_i)) &\lesssim&\eta_a^*(H)\varepsilon_{h_k}(\lambda_i),\\
\Phi(M^*(\lambda_i),M^*_{h_k}(\lambda_i))&\lesssim& \eta_a(H)\varepsilon^*_{h_k}(\lambda_i).\label{Error_u_h_1_nagative}
\end{eqnarray}
Then after one correction step, the resultant eigenpair approximation\\
$(\{\lambda_{j,h_{k+1}}\}_{j=i}^{i+m-1},M_{h_{k+1}}(\lambda_i),M^*_{h_{k+1}}(\lambda_i))$
have following error estimates
\begin{eqnarray}
\Theta(M(\lambda_i),M_{h_{k+1}}(\lambda_i)) &\lesssim& \varepsilon_{h_{k+1}}(\lambda_i),\label{Estimate_u_u_h_2}\\
\Theta(M^*(\lambda_i),M^*_{h_{k+1}}(\lambda_i)) &\lesssim& \varepsilon^*_{h_{k+1}}(\lambda_i),\label{Estimate_u_u_h_2_adjoint}\\
\Phi(M(\lambda_i),M_{h_{k+1}}(\lambda_i)) &\lesssim&
\eta_a^*(H) \varepsilon_{h_{k+1}}(\lambda_i),\label{Estimate_u_h_2_Nagative}\\
\Phi(M^*(\lambda_i),M^*_{h_{k+1}}(\lambda_i)) &\lesssim&
\eta_a(H) \varepsilon^*_{h_{k+1}}(\lambda_i),\label{Estimate_u_h_2_Nagative_Adjoint}
\end{eqnarray}
where
$\varepsilon_{h_{k+1}}(\lambda_i):=\eta_a^*(H)\varepsilon_{h_k}(\lambda_i)+\delta_{h_{k+1}}(\lambda_i)$ and
$\varepsilon^*_{h_{k+1}}(\lambda_i):=\eta_a(H)\varepsilon^*_{h_k}(\lambda_i)+\delta^*_{h_{k+1}}(\lambda_i)$.
\end{theorem}
\begin{proof}
{\color{black}
From (\ref{Iteration_Scheme}),  there exist the basis functions $\big\{u_j\big\}_{j=i}^{i+m-1}$ of $M(\lambda_i)$
such that
\begin{eqnarray}
a(u_j,v)&=&b\left(\sum_{k=i}^{i+m-1}p_{jk}(\lambda_i)u_k, v\right),\ \ \ \forall v\in V,
\end{eqnarray}
where $p_{jk}(\cdot)$ denotes a polynomial of degree no more than $\alpha$ for $k=i,\cdots,j$ with $p_{jj}(\lambda_i)=\lambda_i$
and $p_{jk}(\lambda_i)=0$ for $j<k\leq i+m-1$.
We can define a matrix $\mathcal P:=(p_{j+1-i,k+1-i})_{i\leq j,k\leq i+m-1}\in \mathcal C^{q\times q}$ such that
\begin{eqnarray}
a(U,v)=b(\mathcal P U, v),\ \ \ \ \forall v\in V,
\end{eqnarray}
where $U:=(u_i,\cdots, u_{i+m-1})^T$.
It is easy to know  that the matrix $\mathcal P$ is nonsingular providing $\lambda_i\neq 0$.

For each $\widetilde u_{j,h_{k+1}}$, from the definitions of $\Theta(M(\lambda_i),M_{h_k}(\lambda_i))$ and $\Phi(M(\lambda_i),M_{h_k}(\lambda_i))$,
there exist a vector $\mathcal R_j := (c_1,\cdots,c_m)^T\in \mathcal C^{m\times 1}$ such that
\begin{eqnarray}
\|u_{j,h_k}-\mathcal R_j^T U\|_1 &\lesssim & \varepsilon_{h_k}(\lambda_i),
\ \ \quad \ \ \ \ \ \ \ {\rm for}\ j=i,\cdots,i+m-1,\\
\|u_{j,h_k}-\mathcal R_j^TU\|_0 &\lesssim& \eta_a^*(H)\varepsilon_{h_k}(\lambda_i),
\ \ \ \ {\rm for}\ j=i,\cdots,i+m-1.
\end{eqnarray}

For any $v_{h_{k+1}}\in V_{h_{k+1}}$, we have
\begin{eqnarray}\label{Error_Estimate_1}
&&{|}a(\widetilde u_{j,h_{k+1}}-P_{h_{\ell+1}}\mathcal R_j^T\mathcal P^{-1}U,v_{h_{k+1}}){|}
={|}a(\widetilde u_{j,h_{k+1}}-\mathcal R_j^T\mathcal P^{-1}U,v_{h_{k+1}}){|}\nonumber\\
&=&b(u_j^{h_k}-\mathcal R_j^T\mathcal{P}^{-1}\mathcal PU,v_{h_{k+1}})
= {|}b(u_{j}^{h_k}-\mathcal R_j^TU,v_{h_{k+1}}){|}\nonumber\\
&\lesssim& \eta_a^*(H)\varepsilon_{h_k}(\lambda_i)\|v_{h_{k+1}}\|_1, \ \ \ \ {\rm for}\ j=i,\cdots,i+m-1.
\end{eqnarray}
From (\ref{Inf_Sup_Discrete}) and  (\ref{Error_Estimate_1}), the following estimate holds
\begin{eqnarray}
\|\widetilde U_{j,h_{k+1}}-P_{h_{k+1}}\mathcal R_j^T\mathcal P^{-1}U\|_1 &\lesssim&
\eta_a^*(H)\varepsilon_{h_k}(\lambda_i),\nonumber\\
&& \ \ {\rm for}\ j=i,\cdots,i+m-1.
\end{eqnarray}
Combining with the error estimate
\begin{eqnarray}
\|\mathcal R_j^T\mathcal P^{-1}U-P_{h_{\ell+1}}\mathcal R_j^T\mathcal P^{-1}U\|_1
&\lesssim& \delta_{h_{k+1}}(\lambda_i),\nonumber\\
&&\ \ {\rm for}\ j=i,\cdots,i+m-1,
\end{eqnarray}
we have
\begin{eqnarray}\label{Error_Estimates_Tilde_u_J}
\|\widetilde u_{j,h_{k+1}}-\mathcal R_j^T\mathcal P^{-1}U\|_1 &\lesssim&
\eta_a^*(H)\varepsilon_{h_k}(\lambda_i)+\delta_{h_{k+1}}(\lambda_i),\nonumber\\
&&\ \ \ \ \  \ \ {\rm for}\ j=i,\cdots,i+m-1.
\end{eqnarray}
}
After Step 3, from the definition of $V_{H,h_{k+1}}$ and (\ref{Error_Estimates_Tilde_u_J}), we derive
\begin{eqnarray}\label{Definition_Varepsilon_k+1}
&&\sup_{u \in M(\lambda_i),\|u\|_1=1}
\inf_{v_{H,h_{k+1}}\in V_{H,h_{k+1}}}\| u-v_{H,h_{\ell+1}}\|_1\nonumber\\
&\leq& \sup_{u\in M(\lambda_i),\|u\|_1=1}
\inf_{v_{h_{k+1}}\in W_{h_{k+1}}}\|u-v_{h_{k+1}}\|_1\nonumber\\
&\lesssim& \sup_{v_{h_{k+1}}\in W_{h_{k+1}}, \|v_{h_{k+1}}\|_1=1}
\inf_{u\in M(\lambda_i)}\|v_{h_{\ell+1}}-u\|_1\nonumber\\
&\lesssim& \max_{j=i,\cdots,i+m-1}\|\widetilde u_{j,h_{k+1}}-\mathcal R_j^T\mathcal P^{-1} U\|_1\nonumber\\
&\lesssim& \eta_a^*(H)\varepsilon_{h_k}(\lambda_i)+\delta_{h_{k+1}}(\lambda_i),
\end{eqnarray}
where $W_{h_{k+1}}:={\rm span}\{\widetilde u_i^{h_{k+1}},\cdots,\widetilde u_{i+m-1}^{h_{k+1}}\}$.

{ Similarly},
\begin{eqnarray}\label{Definition_Varepsilon_k+1_Adjoint}
&&\sup_{u_*\in M^*(\lambda),\|u\|_1=1}\inf_{v_{H,h_{k+1}}\in V^*_{H,h_{k+1}}}\|u^*-v_{H,h_{k+1}}\|_1\nonumber\\
 &\lesssim& \eta_a(H)\varepsilon^*_{h_k}(\lambda_i)+\delta^*_{h_{k+1}}(\lambda_i).
\end{eqnarray}
Then from the error estimate results stated in Theorem \ref{Error_Estimate_Theorem}
for the eigenvalue problem (see, e.g., \cite[Section 8]{BabuskaOsborn})
and (\ref{Definition_Varepsilon_k+1})-(\ref{Definition_Varepsilon_k+1_Adjoint}),
the following error estimates hold
\begin{eqnarray}
\Theta(M(\lambda_i),M_{h_{k+1}}(\lambda_i))
&\lesssim& \eta_a^*(H)\varepsilon_{h_k}(\lambda_i)+\delta_{h_{k+1}}(\lambda_i),\\
\Theta(M^*(\lambda_i),M^*_{h_{k+1}}(\lambda_i))
&\lesssim& \eta_a(H)\varepsilon^*_{h_k}(\lambda_i)+\delta^*_{h_{k+1}}(\lambda_i).
\end{eqnarray}
These are the desired estimates (\ref{Estimate_u_u_h_2}) and (\ref{Estimate_u_u_h_2_adjoint}).
Furthermore,
\begin{eqnarray}
\Phi(M(\lambda_i),M_{h_{k+1}}(\lambda_i))&\lesssim & \widetilde{\eta}_a^*(H)\sup_{u\in M(\lambda),\|u\|_1=1}
\inf_{v_{H,h_{k+1}}\in V_{H,h_{k+1}}}\|u-v_{H,h_{k+1}}\|_1\nonumber\\
&\leq & \eta_a^*(H)\varepsilon_{h_{k+1}}(\lambda_i),
\end{eqnarray}
where
\begin{eqnarray}
\widetilde{\eta}_a^*(H):=\sup_{f\in V,\|f\|_b=1}\inf_{v_{H,h_{k+1}}\in V_{H,h_{k+1}}}
\|T_*f-v_{H,h_{k+1}}\|_1\leq \eta_a^*(H).
\end{eqnarray}
Then we obtain (\ref{Estimate_u_h_2_Nagative}). A similar argument leads to (\ref{Estimate_u_h_2_Nagative_Adjoint}).
\end{proof}

\section{Multilevel correction scheme}

In this section, we introduce a multilevel correction
scheme based on the {\it One Correction Step} \ref{Correction_Step}.
The method improves accuracy after each correction step, which is different
from the two-grid methods in \cite{Kolman,XuZhou,YangFan}.

\begin{algorithm}\label{Multi_Correction}
Multilevel Correction Scheme
\begin{enumerate}
\item Construct a coarse conforming finite element space $V_{h_1}$ on $\mathcal{T}_{h_1}$
such that $V_H\subset V_{h_1}$
and solve the following two eigenvalue problems:

Find $(\lambda_{h_1},u_{h_1})\in \mathcal{R}\times V_{h_1}$ such that
$b(u_{h_1},u_{h_1})=1$ and
\begin{eqnarray}\label{Initial_Eigen_Problem}
a(u_{h_1},v_{h_1})&=&\lambda_{h_1}b(u_{h_1},v_{h_1}),\ \ \ \ \forall v_{h_1}\in V_{h_1}.
\end{eqnarray}
Find $(\lambda_{h_1},u^*_{h_1})\in \mathcal{R}\times V_{h_1}$ such that
$b(u^*_{h_1},u^*_{h_1})=1$ and
\begin{eqnarray}\label{Initial_Eigen_Problem_Adjoint}
a(v_{h_1}, u^*_{h_1})&=&\lambda_{h_1}b(v_{h_1}, u_{h_1}^*),\ \ \ \ \forall v_{h_1}\in V_{h_1}.
\end{eqnarray}

Choose $2q$ eigenpairs $\{\lambda_{j,h_1},u_{j,h_j}\}_{j=i}^{i+q-1}$ and
$\{\lambda_{j,h_1},u^*_{j,h_j}\}_{j=i}^{i+q-1}$
which approximate the desired
eigenvalue $\lambda_i$ and its geometric eigenspaces of the eigenvalue problem
 (\ref{Initial_Eigen_Problem})
and its adjoint one (\ref{Initial_Eigen_Problem_Adjoint}).
Based on these two geometric eigenspace, we compute the
corresponding algebraic eigenspaces $M_{h_1}(\lambda_i):={\rm space}\big\{u_{i,h_1},\cdots,u_{i+m-1,h_1}\big\}$
and $M^*_{h_1}(\lambda_i):={\rm space}\big\{u^*_{i,h_1},\cdots,u_{i+m-1,h_1}^*\big\}$.
Then do the following correction steps.
\item Construct a series of finer finite element
spaces $V_{h_2},\cdots,V_{h_n}$ on the sequence of nested meshes $\mathcal{T}_{h_2},\cdots,\mathcal{T}_{h_n}$
(cf. \cite{BrennerScott,Ciarlet}).
\item Do $k=1,\cdots,n-1$\\
Obtain new eigenpair approximations
$(\{\lambda_{j,h_{k+1}}\}_{j=i}^{i+m-1},M_{h_{k+1}}(\lambda_i),M^*_{h_{k+1}}(\lambda_i))$
by Algorithm \ref{Correction_Step}
\begin{eqnarray*}
&&\big(\{\lambda_{j,h_{k+1}}\}_{j=i}^{i+m-1},M_{h_{k+1}}(\lambda_i),M^*_{h_{k+1}}(\lambda_i)\big)=\nonumber\\
&&\ \ \ \ \quad\quad {\it
Correction}\big(V_H,\{\lambda_{j,h_k}\}_{j=i}^{i+m-1},M_{h_k}(\lambda_i),M^*_{h_k}(\lambda_i),V_{h_{k+1}}\big).
\end{eqnarray*}
End Do
\end{enumerate}
Finally, we obtain eigenpair approximations
$\big(\{\lambda_{j,h_n}\}_{j=i}^{i+m-1},M_{h_n}(\lambda_i),M^*_{h_n}(\lambda_i)\big)$.
\end{algorithm}

\begin{theorem}
After implementing Algorithm \ref{Multi_Correction}, the resultant
eigenpair approximation $(\{\lambda_{j,h_n}\}_{j=i}^{i+m-1},M_{h_n}(\lambda_i),M^*_{h_n}(\lambda_i))$
 has following error estimates
\begin{eqnarray}
\Theta(M(\lambda_i),M_{h_n}(\lambda_i)) &\lesssim& \varepsilon_{h_n}(\lambda_i),\label{Multi_Correction_Err_fun}\\
\Phi(M(\lambda_i),M_{h_n}(\lambda_i)) &\lesssim&\eta_a^*(H) \varepsilon_{h_n}(\lambda_i),\label{Multi_Correction_Err_fun_Weak}\\
\Theta(M^*(\lambda_i),M^*_{h_n}(\lambda_i)) &\lesssim& \varepsilon^*_{h_n}(\lambda_i),\label{Multi_Correction_Err_fun_Adjoint}\\
\Phi(M^*(\lambda_i),M^*_{h_n}(\lambda_i)) &\lesssim&\eta_a(H) \varepsilon^*_{h_n}(\lambda_i),\label{Multi_Correction_Err_fun_Weak_Adjoint}\\
|\widehat{\lambda}_{i,h_n}-\lambda_i|&\lesssim&\varepsilon_{h_n}(\lambda_i)\varepsilon^*_{h_n}(\lambda_i),
\label{Multi_Correction_Err_eigen}
\end{eqnarray}
where $\widehat{\lambda}_{i,h_n}=\frac{1}{m}\sum_{j=i}^{i+m-1}\lambda_{j,h_n}$,
$\varepsilon_{h_n}(\lambda_i)=\sum_{k=1}^{n}\eta_a^*(H)^{n-k}\delta_{h_k}(\lambda_i)$ and\\
$\varepsilon^*_{h_n}(\lambda_i)=\sum_{k=1}^{n}\eta_a(H)^{n-k}\delta^*_{h_k}(\lambda_i)$.
\end{theorem}
\begin{proof}
First, the following estimates hold
\begin{eqnarray}
\Theta(M(\lambda_i),M_{h_1}(\lambda_i)) &\lesssim& \varepsilon_{h_1}(\lambda_i),\label{Multi_Correction_Err_fun_h_1}\\
\Phi(M(\lambda_i),M_{h_1}(\lambda_i)) &\lesssim&\eta_a^*(h_1) \varepsilon_{h_1}(\lambda_i)\leq \eta_a^*(H) \varepsilon_{h_1}(\lambda_i),\label{Multi_Correction_Err_fun_Weak_h_1}\\
\Theta(M^*(\lambda_i),M^*_{h_1}(\lambda_i)) &\lesssim& \varepsilon^*_{h_1}(\lambda_i),\label{Multi_Correction_Err_fun_Adjoint_h_1}\\
\Phi(M^*(\lambda_i),M^*_{h_1}(\lambda_i)) &\lesssim&\eta_a(h_1) \varepsilon^*_{h_1}(\lambda_i)\leq \eta_a(H) \varepsilon^*_{h_1}(\lambda_i).\label{Multi_Correction_Err_fun_Weak_Adjointh_1}
\end{eqnarray}
Then we set $\varepsilon_{h_1}(\lambda_i):= \delta_{h_1}(\lambda_i)$ and $\varepsilon_{h_1}^*(\lambda_i):=\delta_{h_1}^*(\lambda_i)$.

By recursive relation and Theorem \ref{Error_Estimate_One_Correction_Theorem},
{ we derive}
\begin{eqnarray}\label{epsilon_n_1}
\Theta(M(\lambda_i),M_{h_n}(\lambda_i))&\lesssim&\varepsilon_{h_n}(\lambda_i) = \eta_a^*(H)\varepsilon_{h_{n-1}}(\lambda_i)
+\delta_{h_n}(\lambda_i)\nonumber\\
&\lesssim&\eta_a^*(H)^2\varepsilon_{h_{n-2}}(\lambda_i)+
\eta_a^*(H)\delta_{h_{n-1}}(\lambda_i)+\delta_{h_n}(\lambda_i)\nonumber\\
&\lesssim&\sum\limits_{k=1}^n\eta_a^*(H)^{n-k}\delta_{h_k}(\lambda_i)
\end{eqnarray}
and
\begin{eqnarray}\label{Error_n-1_Negative_norm}
\Phi(M(\lambda_i),M_{h_n}(\lambda_i))&\lesssim& \eta_a^*(H)\sum\limits_{k=1}^n\eta_a^*(H)^{n-k}\delta_{h_k}(\lambda_i).
\end{eqnarray}
These are the estimates (\ref{Multi_Correction_Err_fun}) and (\ref{Multi_Correction_Err_fun_Weak}) and the
estimates (\ref{Multi_Correction_Err_fun_Adjoint}) and (\ref{Multi_Correction_Err_fun_Weak_Adjoint}) can be proved similarly.
From Theorem \ref{Error_Estimate_Theorem}, (\ref{Multi_Correction_Err_fun}) and (\ref{Multi_Correction_Err_fun_Adjoint}),
we can obtain the estimate (\ref{Multi_Correction_Err_eigen}).
\end{proof}

\section{Numerical results}
In this section, we give some numerical results to illustrate the
efficiency of the multilevel correction scheme defined by Algorithm \ref{Multi_Correction}.
Here, we solve the following eigenvalue problem
\begin{equation}\label{Numerical_Exam_1}
\left\{
\begin{array}{rcl}
-\Delta u+\mathbf b\cdot \nabla u &=& \lambda u,\ \ \ {\rm in}\ \Omega,\\
u&=&0,\ \ \ \ \ {\rm on}\ \partial\Omega,
\end{array}
\right.
\end{equation}
where $\mathbf b=[b_1,b_2]^T\in\mathcal{C}^2$ is a constant vector and $\Omega=(0,1)\times(0,1)$.
This example comes from \cite{HeuvelineRannacher_2001,HeuvelineRannacher_2006}.
We choose $b_1=1$ and $b_2=1/2$ in Subsections \ref{Multi_Space_Subsection} and
\ref{Multi_Grid_Subsection}. Then we choose
$b_1=\cos(\pi x_1)\sin(\pi x_2)$ and $b_2=-\sin(\pi x_1)\cos(\pi x_2)$ in Subsection \ref{Multi_Level_L_Shape}.
We also choose a complex vector $\mathbf b$ in the final example.

When $b_1=1$ and $b_2=1/2$, the problem (\ref{Numerical_Exam_1})
 is nonself-adjoint, but all of its eigenvalues are nondefective (all
algebraic eigenfunctions are of order $1$) and real numbers
\begin{eqnarray}
\lambda_{k,\ell}=\frac{b_1^2+b_2^2}{4}+(k^2+\ell^2)\pi^2,
\end{eqnarray}
for $k,\ell\in \mathcal{N}^+$.

The corresponding eigenfunctions can be chosen as real functions
\begin{eqnarray}
u_{k,\ell} &=& \exp\Big(\frac{b_1x_1+b_2x_2}{2}\Big)\sin(k\pi x_1)\sin(\ell\pi x_2).
\end{eqnarray}
The corresponding adjoint eigenvalue problem has eigenvalues $\lambda_{k,\ell}$ and eigenfunctions
\begin{eqnarray}
u_{k,\ell}^* &=& \exp\Big(-\frac{b_1x_1+b_2x_2}{2}\Big)\sin(k\pi x_1)\sin(\ell\pi x_2).
\end{eqnarray}

\subsection{Multi-space way}\label{Multi_Space_Subsection}
In this case, finer finite element spaces are constructed by
increasing polynomial degrees of the beginning finite element space on the same mesh. We first solve
the eigenvalue problem (\ref{Eigenvalue_Problem_Weak_Discrete}) { by linear finite element
on a relatively coarser mesh $\mathcal{T}_H$, then perform the first correction step
with quadratic element, followed by cubic} element for the second correction step
and quartic element for the third correction step. Our initial mesh $\mathcal{T}_H$
is obtained from the Delaunay triangulation followed by four levels of regular mesh refinement.
Figure \ref{Error_First_Eigenvalue_Multi_Space} depicts errors for the
first eigenvalue ($5/16+2\pi^2$) approximation, and
Figure \ref{Error_First_Eigenfunction_Multi_Space} plots numerical errors for
 the eigenfunction and the corresponding adjoint eigenfunction associated with the first eigenvalue.

\begin{figure}[ht]
\centering
\includegraphics[width=7cm,height=5cm]{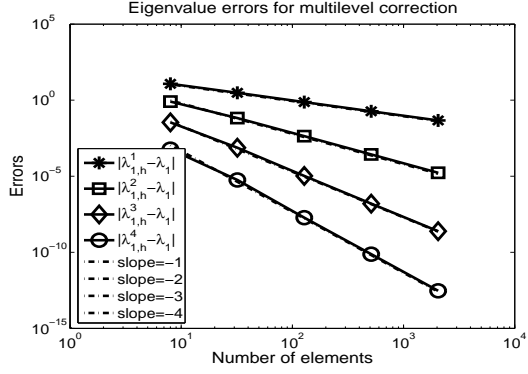}
\caption{\it {  Here,
$\lambda_h^1$ denote the eigenvalue approximation by linear element,
$\lambda_h^2$ is the eigenvalue approximation by the first correction with quadratic element,
$\lambda_h^3$ the eigenvalue approximation by the second correction with cubic element,
$\lambda_h^4$ the eigenvalue approximation by the third correction with quartic element} }
\label{Error_First_Eigenvalue_Multi_Space}
\end{figure}
\begin{figure}[ht]
\centering
\includegraphics[width=6cm,height=5cm]{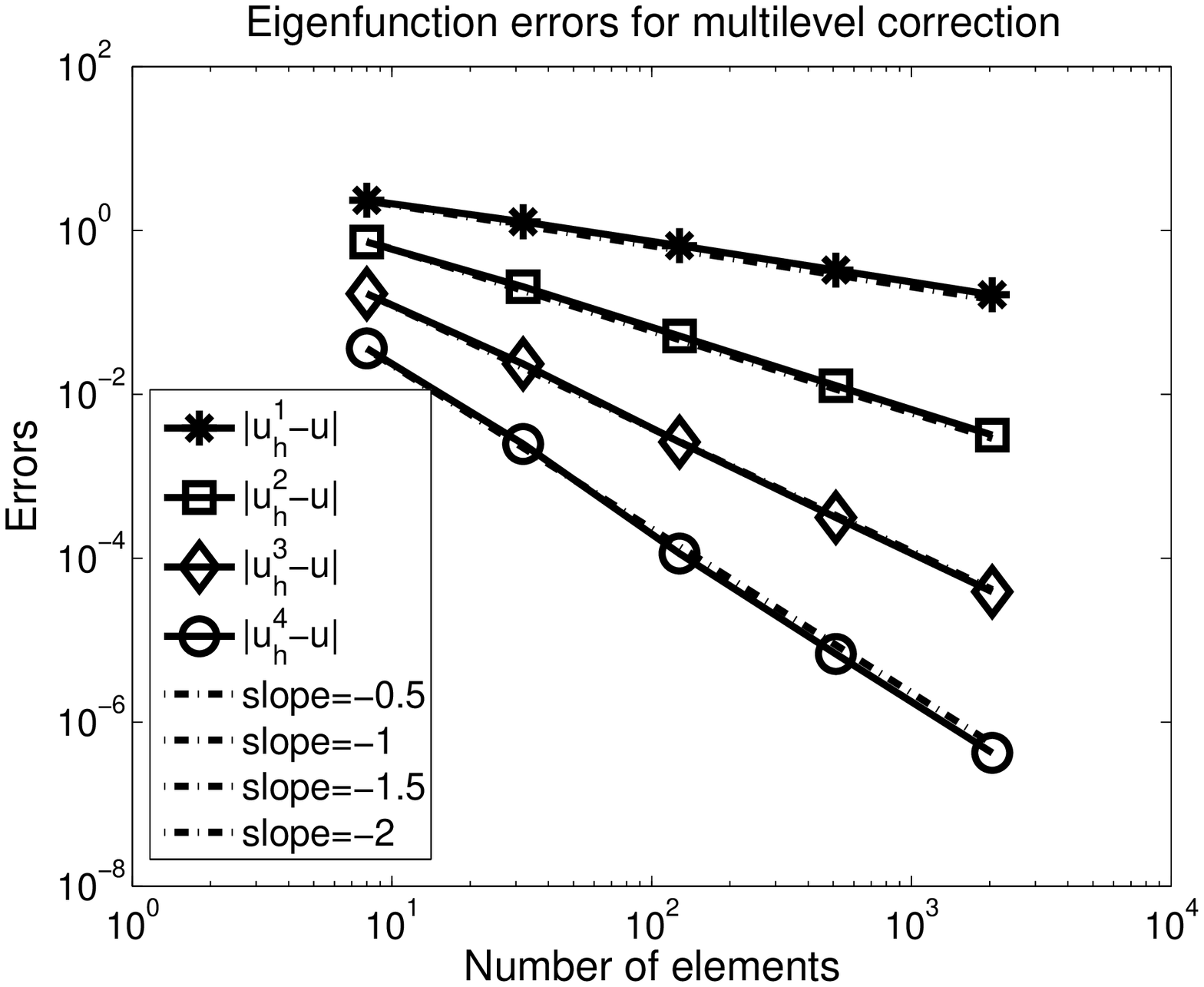}
\includegraphics[width=6cm,height=5cm]{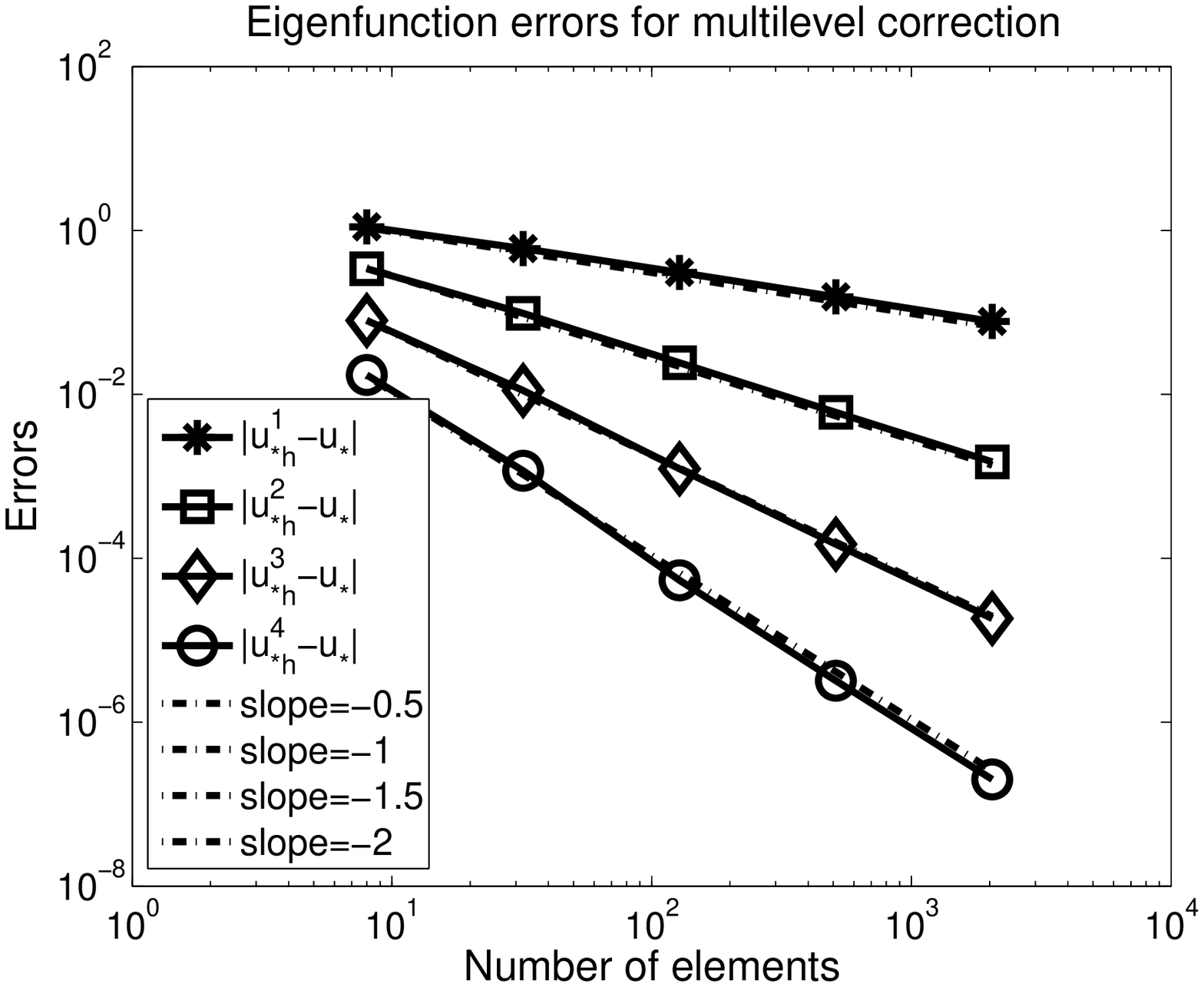}
\caption{\it {Here, $u_h^1$ and $u_{*h}^1$ denote the eigenfunction approximation and its
adjoint approximation by linear element,
$u_h^2$ and $u_{*h}^2$ are eigenfunction approximation and its adjoint approximation
by the first correction with quadratic element, $u_h^3$ and $u_{*h}^3$,
eigenfunction and its adjoint approximation by the second correction with cubic element,
$u_h^4$ and $u_{*h}^4$, eigenfunction and its adjoint approximation by the third
correction with quartic element}\label{Error_First_Eigenfunction_Multi_Space}}
\end{figure}
Furthermore, Figure \ref{Error_First_6_Eigenvalues_Multi_Space} provides numerical results
for the summation of the errors for the first $6$ eigenvalues: $5/16+[2\pi^2,5\pi^2,5\pi^2,8\pi^2,10\pi^2,10\pi^2]$.
\begin{figure}[ht]
\centering
\includegraphics[width=7cm,height=5cm]{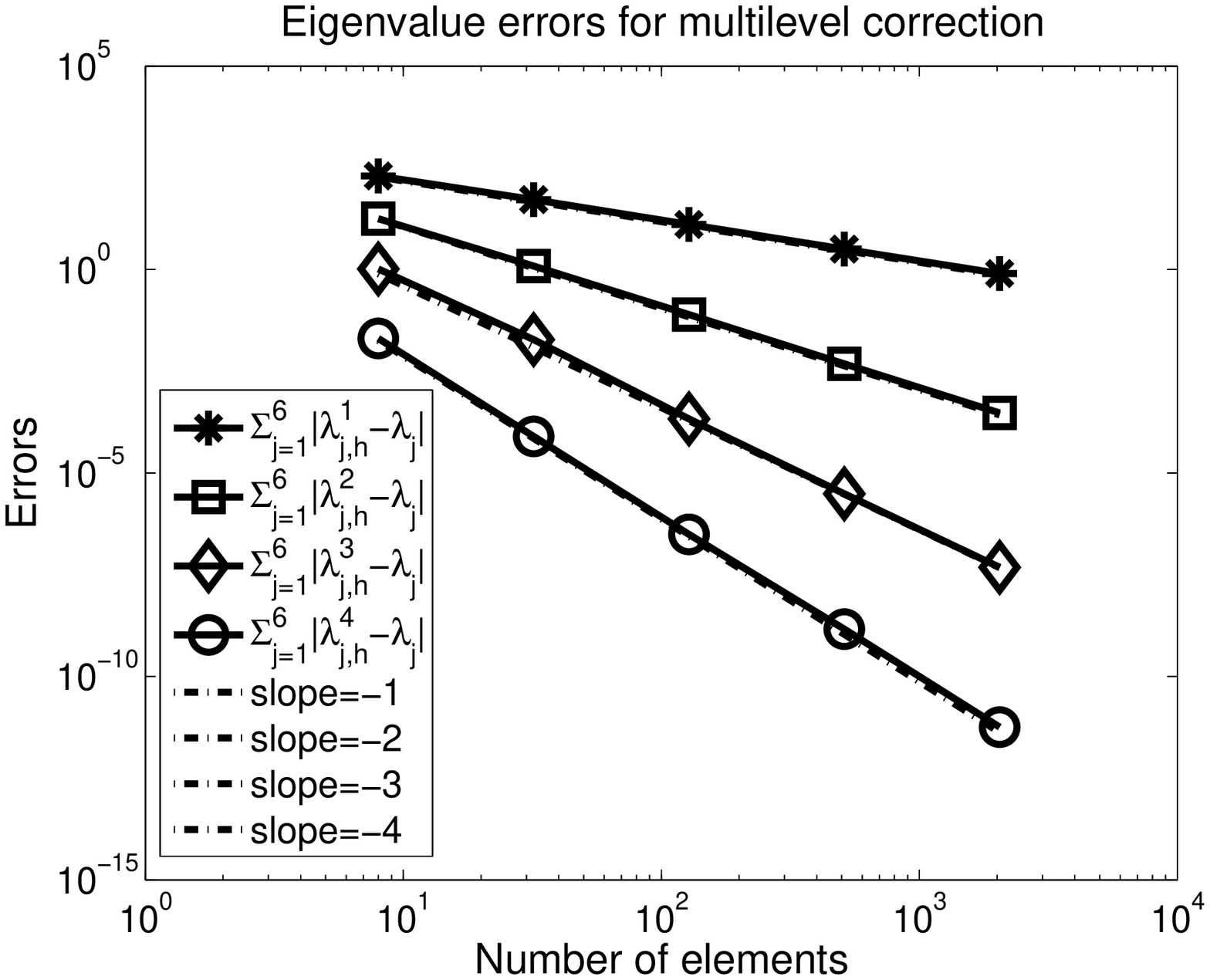}
\caption{\it {Approximation errors for the summation of the errors for the first $6$ eigenvalues by the multi-space way.
Here, $\lambda_{j,h}^1$ denote the eigenvalue approximation by linear element,
$\lambda_{j,h}^2$ is the eigenvalue approximation by the first correction with quadratic element,
$\lambda_{j,h}^3$ the eigenvalue approximation by the second correction with cubic element,
$\lambda_{j,h}^4$ the eigenvalue approximation by the third correction with quartic element}}
 \label{Error_First_6_Eigenvalues_Multi_Space}
\end{figure}

From Figures \ref{Error_First_Eigenvalue_Multi_Space}-\ref{Error_First_6_Eigenvalues_Multi_Space},
 we find that each correction step improves the convergence order by two for
 eigenvalue approximation, and by one for eigenfunction approximation
when the exact eigenfunction is sufficiently smooth.

To end this subsection, we make a comparison with the PPR method \cite{NagaZhang}.
We see from Figure \ref{Comparison_PPR}
that the two-level correction scheme by the multi-space way
 has slightly better accuracy than the PPR method. However,
the two-level correction needs to solve two extra boundary value problems while
the PPR method only need to perform a local recovery at each node. Thus, we
should say that the PPR method has better efficiency than the two-level
correction under regular mesh refinement
when the eigenfunction has regularity $H^3(\Omega)\cap W^{2,\infty}(\Omega)$. Nevertheless,
three and four-level correction will outperform the PPR method.
\begin{figure}[ht]
\centering
\includegraphics[width=7cm,height=5cm]{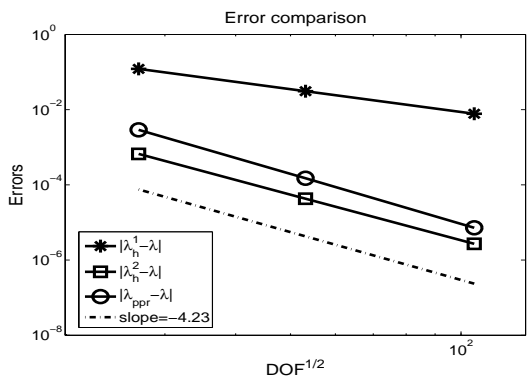}
\caption{\it { Comparison with the PPR method in \cite{NagaZhang} when
$b_1=10$ and $b_2=1$.
$\lambda_h^1$, eigenvalue approximation by linear element;
$\lambda_h^2$, eigenvalue approximation by the first correction with quadratic element;
$\lambda_{\rm ppr}$, the eigenvalue approximation by the PPR method}} \label{Comparison_PPR}
\end{figure}

\subsection{Multi-grid way}\label{Multi_Grid_Subsection}
An alternative way of the multilevel correction
scheme is to construct finer finite element spaces by mesh refinement.
We first solve the eigenvalue problem
(\ref{Eigenvalue_Problem_Weak_Discrete}) in the linear finite element space on
 an initial coarse mesh
$\mathcal{T}_H$ ($\mathcal{T}_{h_1}:=\mathcal{T}_H$).
Then we refine the mesh regularly with the resultant meshes $\mathcal{T}_{h_k}$
satisfying $h_k=2^{1-k}H$ for $(k=2,\cdots,n)$,
 and solve auxiliary source problems (\ref{aux_problem}) and (\ref{aux_problem_Adjoint})
 in the linear finite element space $V_{h_k}$ defined on $\mathcal{T}_{h_k}$ and the
 corresponding eigenvalue problems (\ref{Eigen_Augment_Problem}) and
 (\ref{Eigen_Augment_Problem_Adjoint}) in $V_{H,h_k}$.
We have the following estimate
\begin{eqnarray*}
\varepsilon_{h_n}(\lambda)&=&\sum\limits_{k=1}^{n}H^{n-k}h_k
=\sum\limits_{k=1}^{n}(2H)^{n-k}h_n\leq\frac{1}{1-2H}h_n\approx h_n,
\end{eqnarray*}
and similarly $\varepsilon^*_{h_n}(\lambda)\approx h_n$,
which implies that the multilevel correction method achieves the optimal
convergence rate if the initial mesh size $H$ is reasonably small,
say $H=1/4$ as we will use in our numerical tests.

Numerical results for the first eigenvalue $\lambda=5/16+2\pi^2$
and the two associated eigenfunctions are demonstrated in Figures
\ref{Error_First_Eigenvalue_Multi_Grid} and \ref{Error_First_Eigenfunction_Multi_Grid}, respectively.
Here we use the uniform meshes with $H=1/4$.
\begin{figure}[ht]
\centering
\includegraphics[width=7cm,height=5cm]{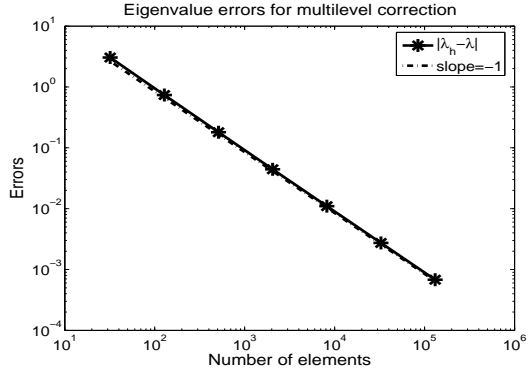}
\caption{\it { Approximation errors} for the first eigenvalue $5/16+2\pi^2$
by the multi-grid way with $H=1/4$}
\label{Error_First_Eigenvalue_Multi_Grid}
\end{figure}
\begin{figure}[ht]
\centering
\includegraphics[width=6cm,height=5cm]{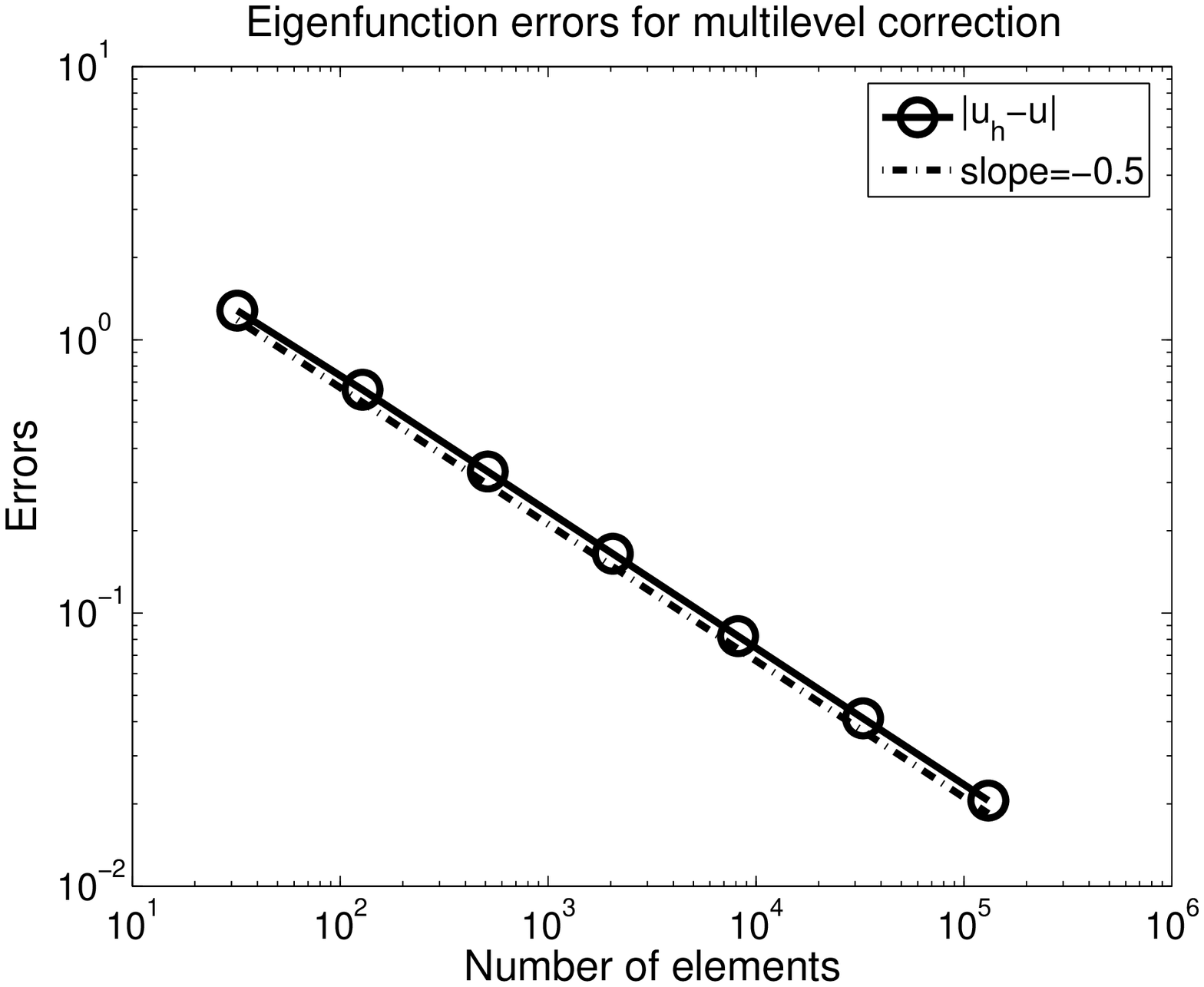}
\includegraphics[width=6cm,height=5cm]{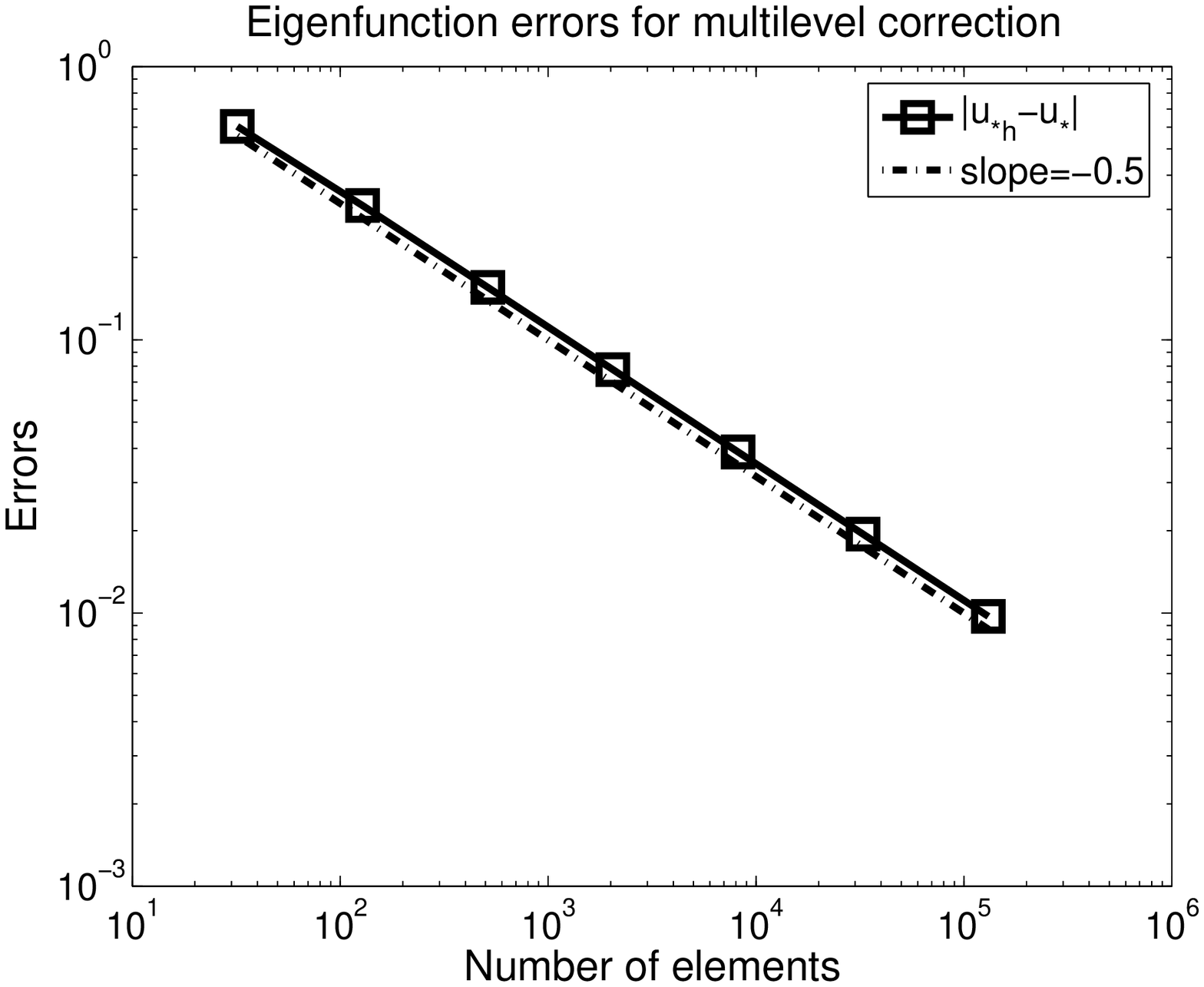}
\caption{\it { Approximation errors for the first eigenfunction and its adjoint
by the multi-grid way with $H=1/4$}}
\label{Error_First_Eigenfunction_Multi_Grid}
\end{figure}
Furthermore, Figure \ref{Error_First_6_Eigenvalues_Multi_Grid}
 provides numerical results for the summation of the errors for the first
$6$ eigenvalues: $5/16+[2\pi^2,5\pi^2,5\pi^2,8\pi^2,10\pi^2,10\pi^2]$
with $H=1/8$ and $H=1/16$, respectively.
\begin{figure}[ht]
\centering
\includegraphics[width=6cm,height=5cm]{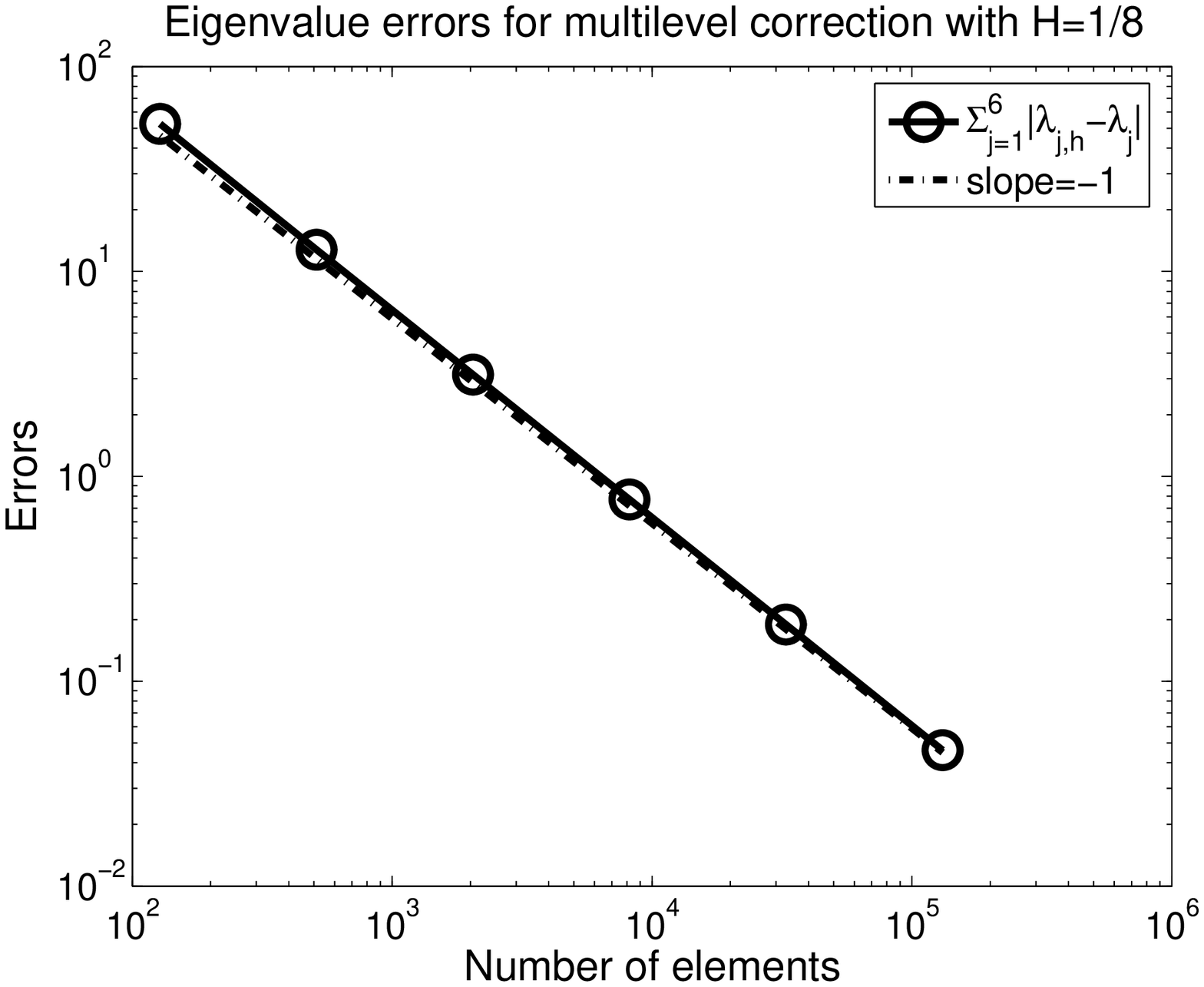}
\includegraphics[width=6cm,height=5cm]{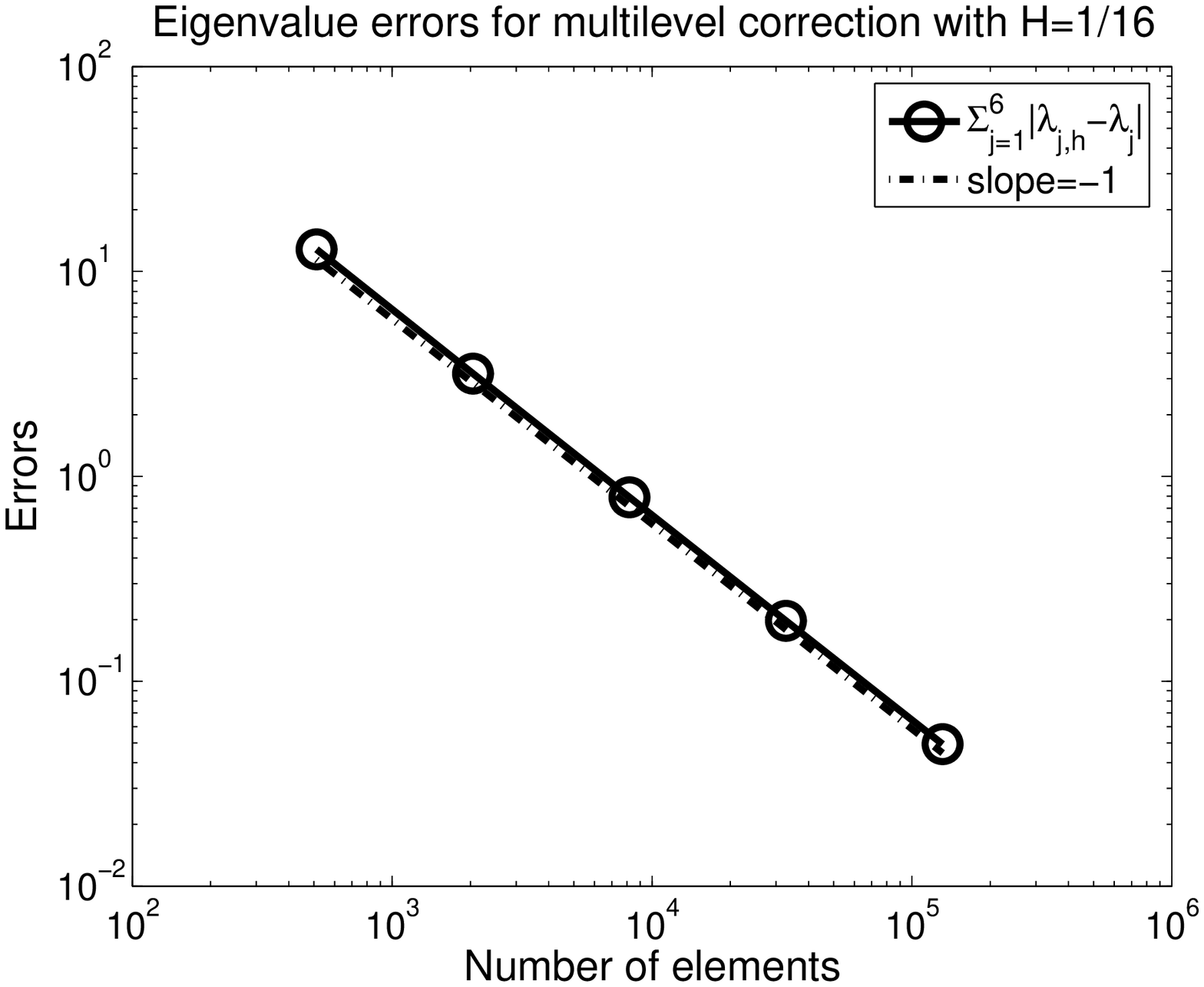}
\caption{\it {Approximation errors for the error summation of the first $6$ eigenvalues by
the multi-grid way with $H=1/8$ (left) and $1/16$ (right)}}
\label{Error_First_6_Eigenvalues_Multi_Grid}
\end{figure}

We observe from Figures \ref{Error_First_Eigenvalue_Multi_Grid}-\ref{Error_First_6_Eigenvalues_Multi_Grid},
that our multilevel correction method with
the multi-grid way produces eigenvalue and eigenfunction approximations with the
optimal convergence rate. Therefore, we can combine the multigrid method for boundary
value problems and our multilevel correction scheme (cf. \cite{LinXie,Xie_IMA})
to achieve better efficiency for nonsymmetric eigenvalue problems.

\subsection{Eigenvalue problem on $L$-shape domain}\label{Multi_Level_L_Shape}
In this subsection, we consider the eigenvalue problem (\ref{Numerical_Exam_1})
on the $L$-shape domain
 $\Omega=(-1,1)\times(-1,1)\backslash[0, 1)\times (-1, 0]$.
Since $\Omega$ has a reentrant corner, the singularity of eigenfunctions is expected.
As a consequence, the convergence rate for the first eigenvalue approximation is $4/3$
by the linear finite element method on quasi-uniform meshes.
Since the exact eigenvalue is unknown, we choose an adequately accurate approximation
$\lambda = 9.95240442893276$ as the exact first eigenvalue for our numerical tests.

Our multilevel correction scheme is tested on a sequence of
meshes $\mathcal{T}_H$ ($\mathcal{T}_{h_1}:=\mathcal{T}_H$), $\mathcal{T}_{h_2}, \cdots, \mathcal{T}_{h_n}$
 produced by the adaptive refinement
(cf. \cite{WuZhang,XuZhou_Eigen}). Here the ZZ recovery method (cf. \cite{ZienkiewiczZhu}) is adopted as
the {\it a posteriori} error estimator for eigenfunction and adjoint eigenfunction
 approximations $\sqrt{\|u_h-u\|_{a,h}^2+\|u_h^*-u^*\|_{a,h}^2}$.
Figure \ref{Mesh_AFEM_Exam_2} shows the initial mesh and the mesh after $12$
adaptive iterations. Figure \ref{Convergence_AFEM_Exam_2_First}
 gives the corresponding numerical results for the adaptive iterations.
\begin{figure}[ht]
\centering
\includegraphics[width=5.5cm,height=5.5cm]{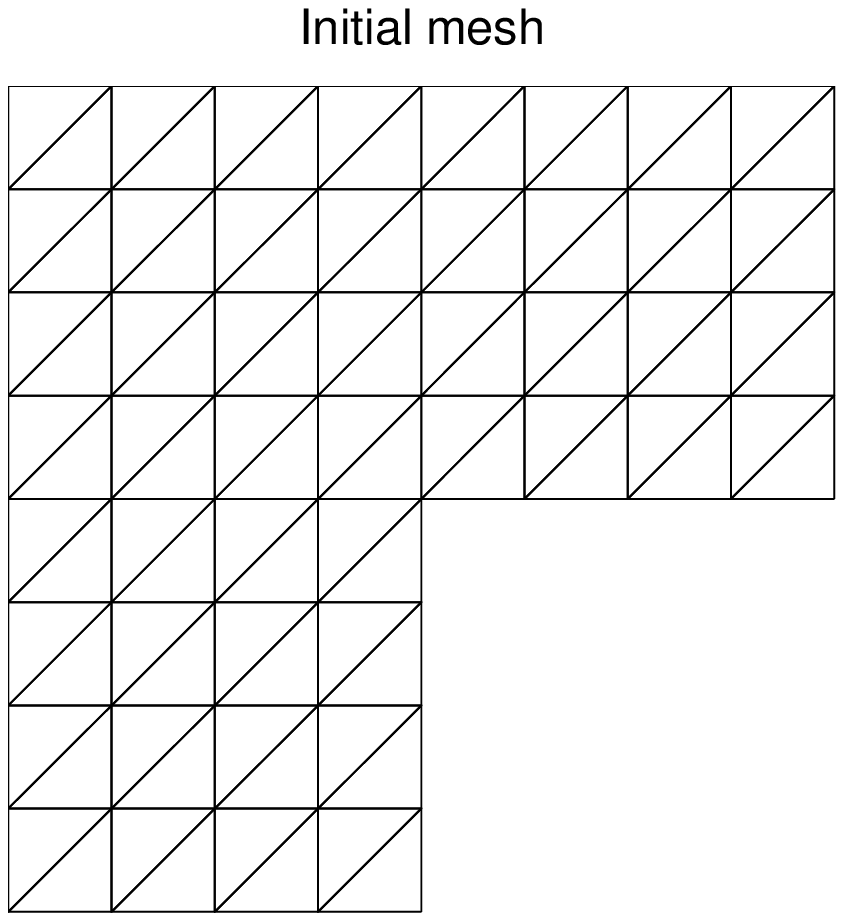}
\includegraphics[width=5.5cm,height=5.5cm]{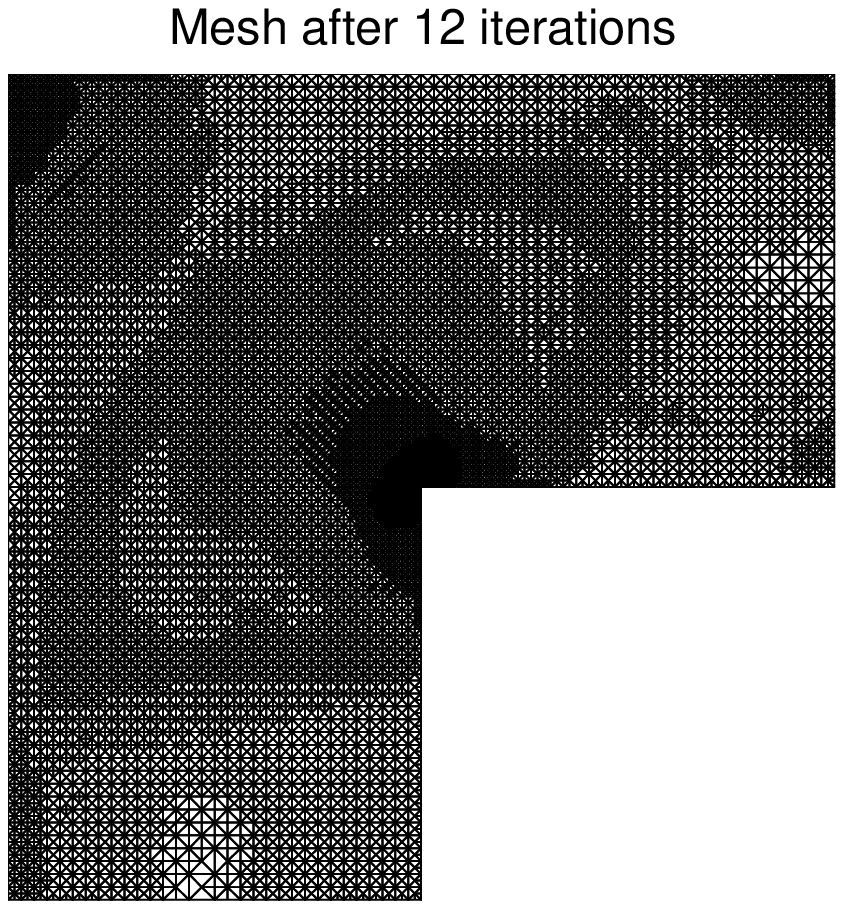}
\caption{\it The initial mesh and the one after 12 adaptive iterations for
 the L-shape domain} \label{Mesh_AFEM_Exam_2}
\end{figure}
\begin{figure}[ht]
\centering
\includegraphics[width=6cm,height=6cm]{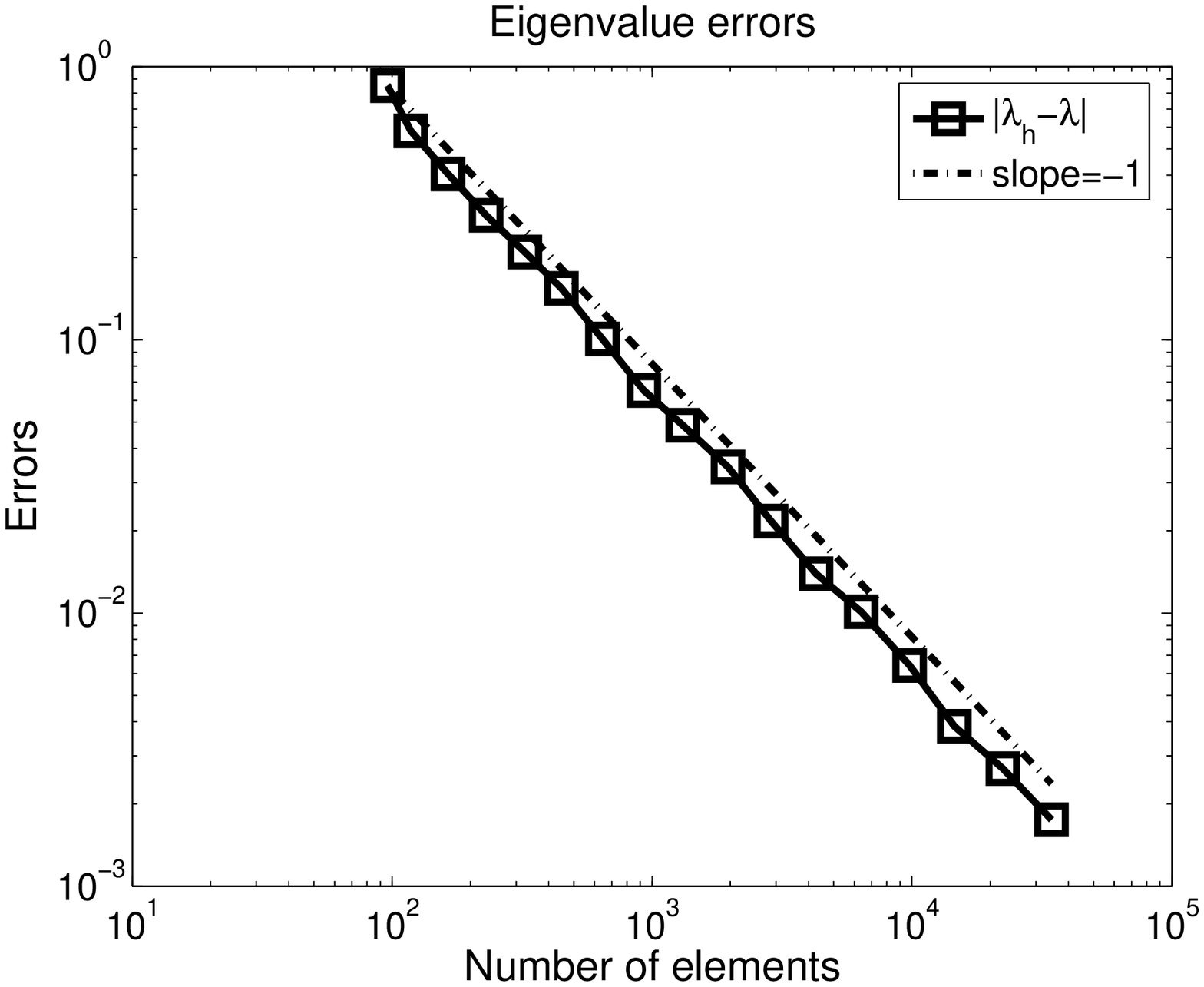}
\includegraphics[width=6cm,height=6cm]{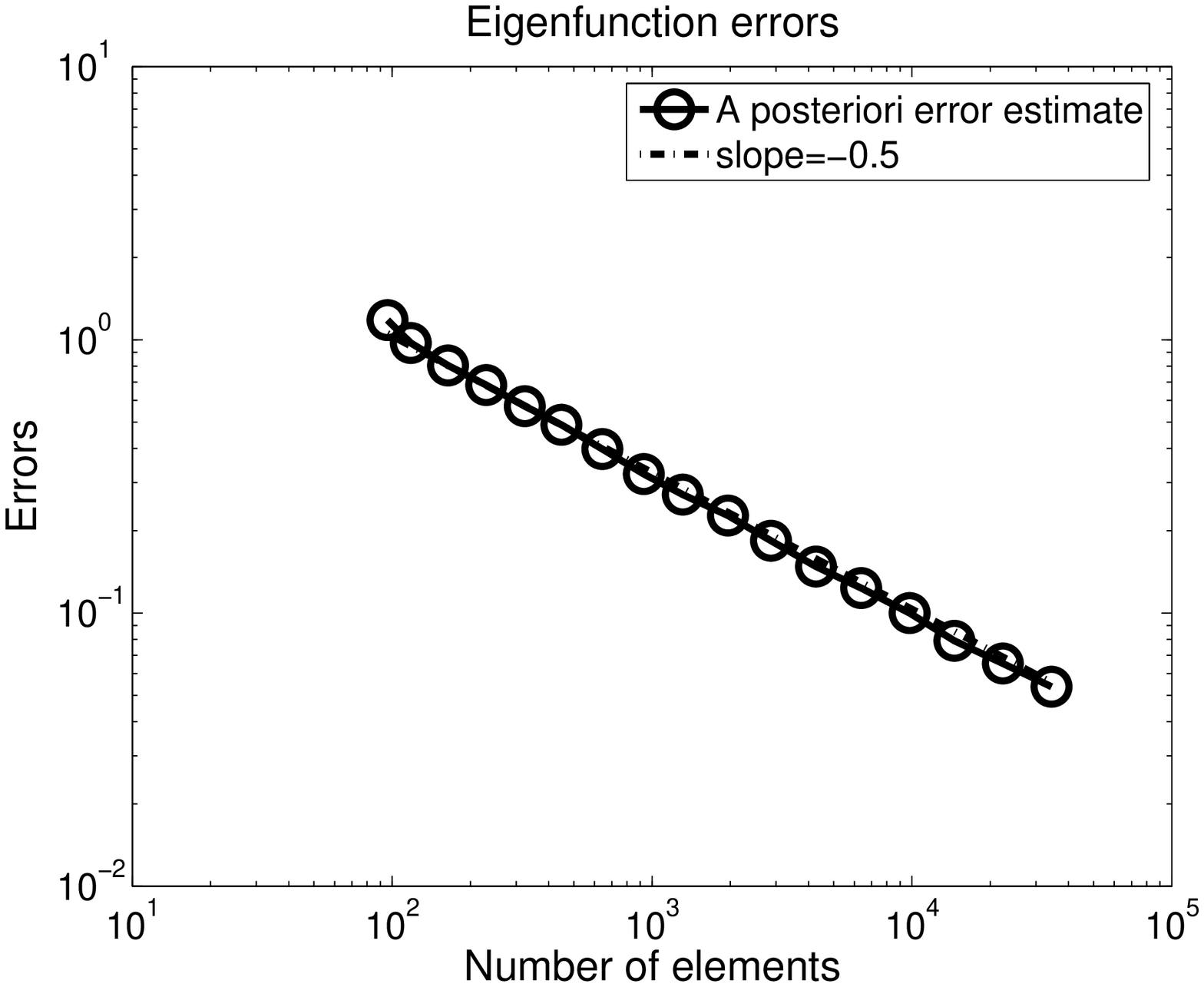}
\caption{\it { Approximation errors of the first eigenvalue and the a
posteriori errors of the associated eigenfunction and adjoint eigenfunction}} \label{Convergence_AFEM_Exam_2_First}
\end{figure}

From Figure \ref{Convergence_AFEM_Exam_2_First}, we observe that the multilevel correction method
works well on adaptive meshes with the optimal convergence rate. Furthermore, the
situation is very different from the two-gird \cite{Kolman,XuZhou,YangFan} method in that the initial
mesh has very little impact on the finest one.
Thus the multilevel correction method can be coupled with the adaptive refinement naturally.

\subsection{Eigenvalue problem with complex vector}
In this subsection, we test the multilevel correction scheme for the problem (\ref{Numerical_Exam_1})
with complex vector $\mathbf b=[1+2\textrm{i}, 1/2-\textrm{i}]^T$. We use the multi-space and multi-grid
ways as in Subsections \ref{Multi_Space_Subsection} and \ref{Multi_Grid_Subsection}, respectively,
 to check the multilevel correction scheme. Figure \ref{Error_First_6_Eigenvalues_Complex}
 shows the numerical results for the first $6$ eigenvalues. It is observed from
 Figure \ref{Error_First_6_Eigenvalues_Complex} that the multilevel correction method defined in
 Algorithm \ref{Multi_Correction} can also work very well for the nonsymmetric eigenvalue problems
 with complex vector.

 \begin{figure}[ht]
\centering
\includegraphics[width=6cm,height=5cm]{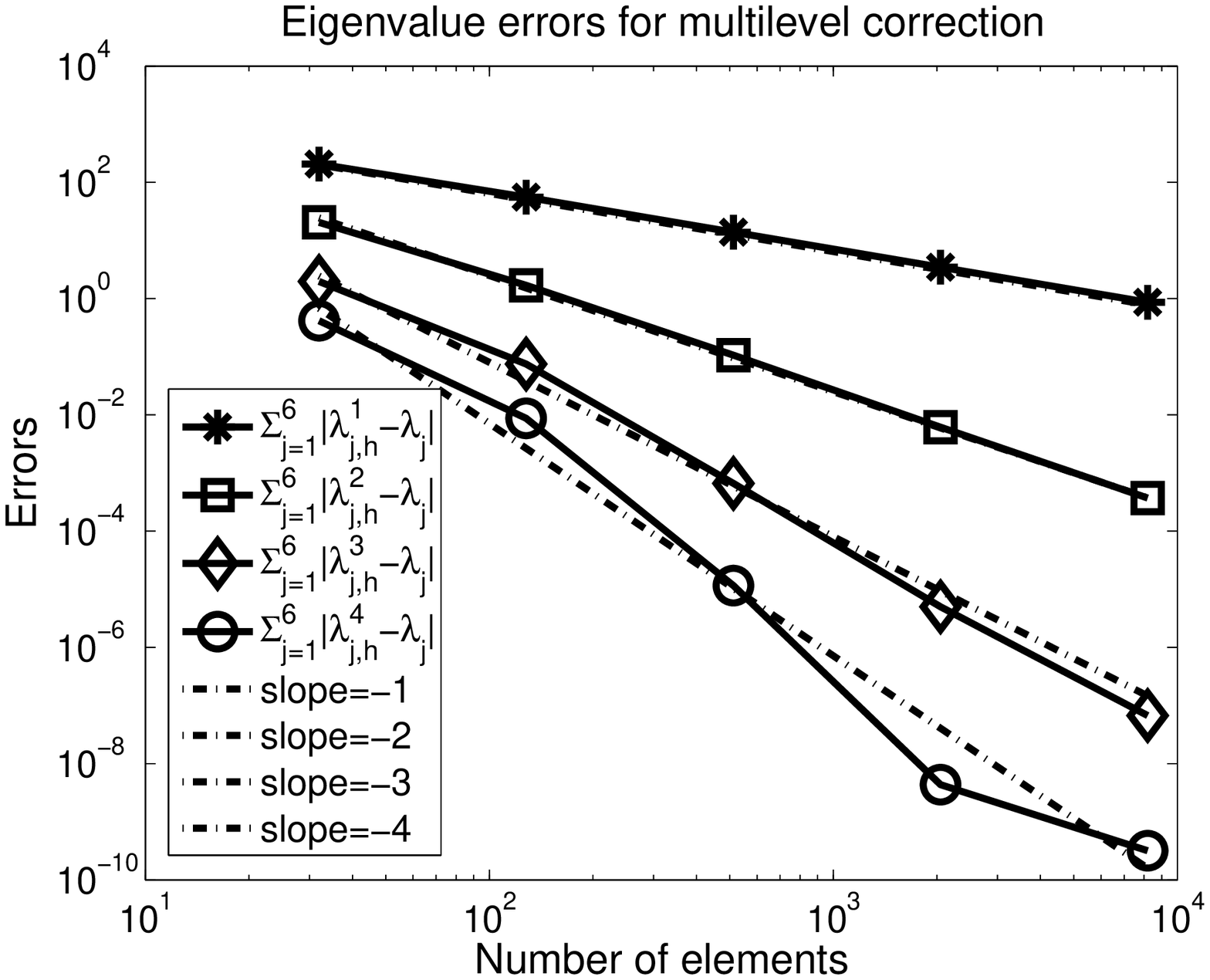}
\includegraphics[width=6cm,height=5cm]{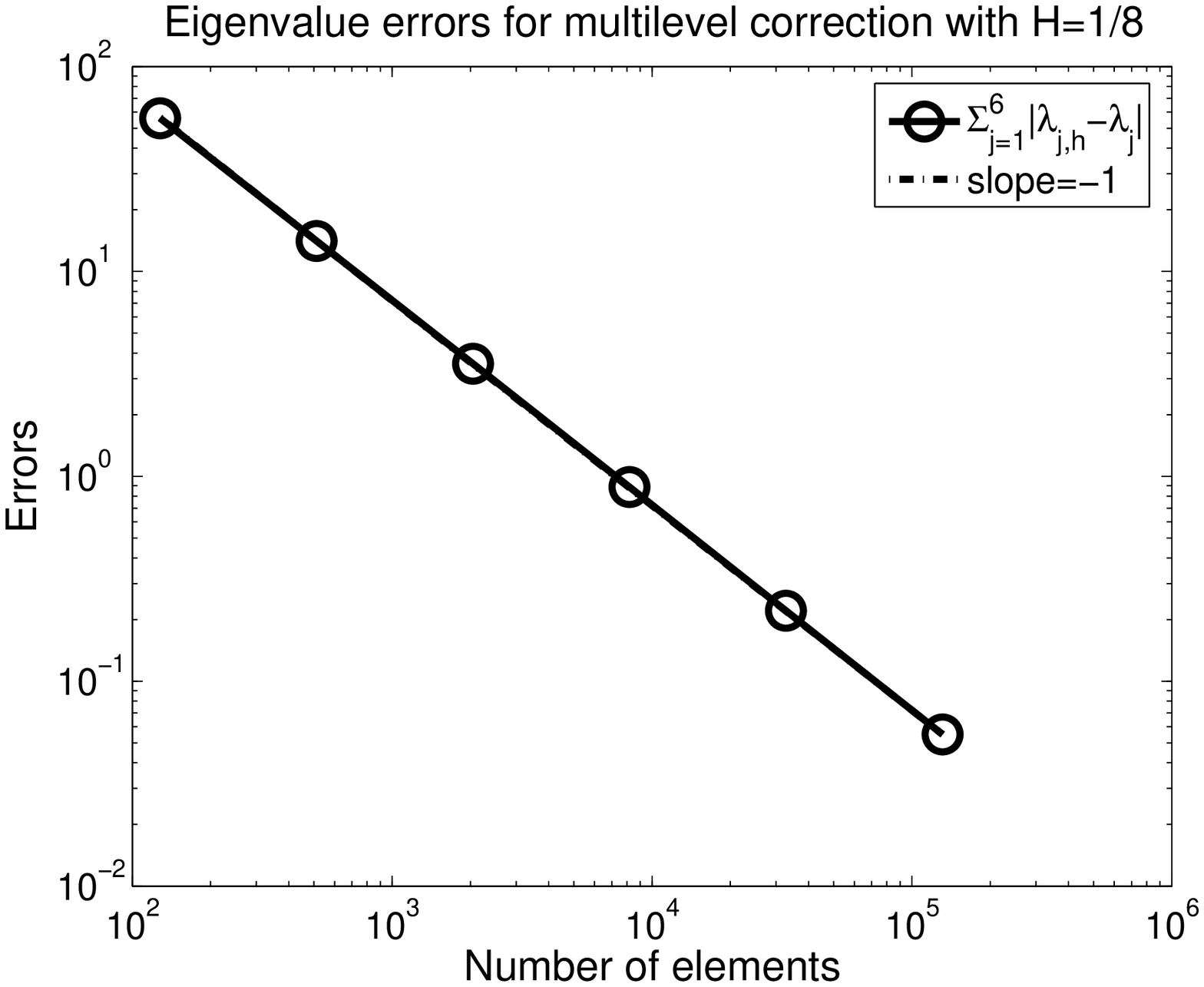}
\caption{\it { Approximation errors for the summation of errors for the first $6$ eigenvalues by
the multi-space way (left) and  the multi-grid way with $H=1/8$ (right). Here,
$\lambda_{j,h}^1$ denote the eigenvalue approximation by linear element,
$\lambda_{j,h}^2$ is the eigenvalue approximation by the first correction with quadratic element,
$\lambda_{j,h}^3$ the eigenvalue approximation by the second correction with cubic element,
$\lambda_{j,h}^4$ the eigenvalue approximation by the third correction with quartic element}}
\label{Error_First_6_Eigenvalues_Complex}
\end{figure}

\section{Concluding remarks}
In this paper, we propose and analyze a multilevel correction scheme
to improve the efficiency of both {\color{black}defective and nondefective} nonsymmetric eigenpair approximations.
In this multilevel correction, we only need to solve eigenvalue problems in the coarsest
 finite element space. Sometimes, we also need to compute the algebraic eigenspace based on the
geometric eigenspace when the ascent is larger than $1$.

Furthermore, our multilevel correction scheme can be coupled with
the multigrid method to construct a parallel
method for eigenvalue problems (see, e.g, \cite{LinXie,LinXie_Multigrid,Xie_IMA,XuZhou_Eigen}).
It can also be combined with  adaptive techniques (cf. \cite{WuZhang}) for singular
 eigenfunction cases. These will be our future work.

{\bf A final remark.} As long as higher eigenvalues are concerned, the multi-space way is preferred
(than the multi-grid way). We can see it clearly by comparing numerical accuracies for summations
of the first 6 eigenvalues in \S 5.1 and \S 5.2.


\section*{Acknowledgments}
The first author is supported in part by the National Natural Science Foundation of
China (NSFC 91330202, 11001259, 11371026, 11201501, 11031006, 2011CB309703)
and the National Center for Mathematics and Interdisciplinary Science,
CAS and the President Foundation of AMSS-CAS.
The second author is supported in part by the US National Science Foundation
 through grant DMS-1115530, DMS-1419040, and the National Natural Science Foundation
  of China (91430216, 11471031).


\bibliographystyle{amsplain}

\end{document}